\documentclass[a4paper]{amsart}
\usepackage[utf8]{inputenc}
\usepackage[T1]{fontenc}
\usepackage{lmodern}
\usepackage{amsmath, amsthm, amssymb, amsfonts}
\usepackage[all]{xy}
\usepackage{nicefrac,mathtools}
\usepackage[shortlabels, inline]{enumitem}
\usepackage{microtype}
\usepackage{hyperref}
\usepackage{graphicx}
\usepackage[all]{xy}
\usepackage{xcolor}
\usepackage{tikz-cd}
\usepackage{thmtools}

\numberwithin{equation}{section}
\theoremstyle{plain}
\newtheorem{theorem}[equation]{Theorem}

\newtheorem{lemma}[equation]{Lemma}

\newtheorem{proposition}[equation]{Proposition}

\newtheorem{corollary}[equation]{Corollary}

\theoremstyle{definition}
\newtheorem{definition}[equation]{Definition}

\theoremstyle{remark} \newtheorem{remark}[equation]{Remark}
 \newtheorem*{remark*}{Remark}
\newtheorem*{remarks*}{Remark} 
\newtheorem{obs}[equation]{Observation}

\newcommand*{\Contc}{\mathrm{C_c}}
\newcommand*{\Cc}{\mathrm{C_c}}
 \newcommand{\base}[1][G]{{#1}^{(0)}}
\newcommand{\dd}{\mathrm{d}}

\newcommand{\id}{\textup{id}}
\newcommand{\supp}{{\textup{supp}}}

\newcommand*{\Hilm}[1][H]{\mathcal #1}
\newcommand{\Bound}{\mathbb B}
\newcommand*{\Cst}{\mathrm{C}^*}
\newcommand{\A}{\mathcal{A}}
\newcommand{\B}{\mathcal{B}}
\newcommand*{\inpro}[2]{\langle#1, #2\rangle}
\newcommand{\norm}[1]{\lvert\!\lvert #1\rvert\!\rvert}
\newcommand{\C}{\mathbb{C}}
\newcommand{\R}{\mathbb{R}}
\newcommand{\Z}{\mathbb{Z}}
\newcommand{\N}{\mathbb{N}}

\newcommand*{\nb}{\nobreakdash}
\newcommand*{\defeq}{\mathrel{\vcentcolon=}}
\newcommand{\etale}{{\'e}tale}
\renewcommand{\phi}{\varphi}

\title[Inclusions of Fell bundles \(\mathrm{C}^*\)-algebras]{Inclusions of Fell bundles \(\mathrm{C}^*\)-algebras and coaction crossed products}

\author{Md Amir Hossain} \email{mdamirhossain18@gmail.com}

\address{Theoretical Statistics and Mathematics Unit, Indian Statistical Institute, Delhi centre, 7, S. J. S. Sansanwal Marg, New Delhi 110016, India.
}


\begin{document}
\begin{abstract}
 Let $p \colon \mathcal{A} \to G$ be a Fell bundle over a locally compact Hausdorff second countable groupoid $G$ equipped with a Haar system, and let $\Gamma$ be a discrete group. Given a continuous $1$-cocycle $c \colon G \to \Gamma$, we show that the $\mathrm{C}^*$-algebra of the restricted Fell bundle $\mathcal{A}|_{G_e}$ embeds isometrically into $\mathrm{C}^*(G;\mathcal{A})$, where $G_e = c^{-1}(e)$ is the clopen subgroupoid corresponding to the identity element.
 We exploit this embedding to show that $\mathrm{C}^*(G;\mathcal{A})$ admits a natural structure of a topologically graded $\mathrm{C}^*$-algebra in the sense of Exel. As a consequence, we obtain a canonical coaction $\delta$ of $\Gamma$ on $\mathrm{C}^*(G; \mathcal{A})$. We further show that the associated coaction crossed product $\mathrm{C}^*(G; \mathcal{A})\rtimes_\delta \Gamma$ is naturally isomorphic to the $\mathrm{C}^*$-algebra of a Fell bundle constructed from the cocycle data.
\end{abstract}

\maketitle

\section{Introduction}

Fell bundles over groupoids introduced by Kumjian~\cite{Kumjian1998Fell-bundles-over-gpd}, provide a unified framework for encoding a wide range of constructions in operator algebras, including twisted groupoid \(\Cst\)\nb-algebras, higher-rank graph \(\Cst\)\nb-algebras and \(\Cst\)-dynamical systems.

Let \(G\) be a locally compact Hausdorff second countable groupoid equipped with a Haar system, and let \(\Gamma\) be a discrete group. A continuous \(1\)-cocycle \(c \colon G \to \Gamma\) induces a decomposition \(G = \bigsqcup_{\gamma \in \Gamma} G_\gamma\), where \(G_\gamma = c^{-1}(\gamma)\). The identity fibre \(G_e\) plays a distinguished role: its \(\Cst\)\nb-algebra embeds canonically into \(\Cst(G)\) (\cite[Theorem 1.1]{Armstrong-Clark-an-Huef-Inclusion-graded-gpd}), and this inclusion realizes \(\Cst(G)\) as a topologically graded \(\Cst\)-algebra in the sense of Exel~\cite[Chapter 19]{Exel2017Book-Partial-action-Fell-bundle}.

In this paper, we extend this picture to Fell bundles. Let \(p \colon \A \to G\) be a Fell bundle and let \(c \colon G \to \Gamma\) be a continuous \(1\)\nb-cocycle. Restricting the Fell bundle \(\A\) to the fibres \(G_{\gamma}\) yields upper semicontinuous Banach bundles \(\A|_{G_\gamma}\), and in particular a Fell bundle \(\A|_{G_e}\) over \(G_e\).
Our first result shows that the natural inclusion of compactly supported sections extends to an isometric embedding \(
\Cst(G_e;\A|_{G_e}) \hookrightarrow \Cst(G;\A)\)
(see Theorem~\ref{thm-isometric-isomorphism}). The proof proceeds via the construction of a canonical \(\Cst\)\nb-correspondence from \(\Cst(G;\A)\) to \(\Cst(G_e;\A|_{G_e})\) ( see Theorem~\ref{lem-inner-prod-positive}), which allows us to identify the \(\Cst\)\nb-norms of sections supported in \(G_e\) inside both the algebras.

We then show that the subspaces \(\Cc(G_\gamma;\A|_{G_\gamma})\) determine a \(\Gamma\)\nb-grading of \(\Cst(G;\A)\). Moreover, this is a \emph{topological grading} in Exel’s sense: restriction to the clopen subgroupoid \(G_e\) defines a canonical conditional expectation onto \(\Cst(G_e;\A|_{G_e})\).

This topological grading enables us to define a coaction \(\delta\) of the discrete group \(\Gamma\) on \(\Cst(G;\A)\). We study the associated coaction crossed product and show that it admits a concrete realization in terms of a Fell bundle naturally associated to the cocycle data \(c\). More precisely, we prove that the coaction crossed product is isomorphic to the \(\Cst\)\nb-algebra of a \emph{pull-back Fell bundle} \(\phi^*\A\) over the \emph{skew-product groupoid} \(G(c)\), where 
\(\phi\colon G(c)\to G\) is a canonical groupoid homomorphism (see Theorem~\ref{thm-coaction-crossed-product}).

As an application, we consider twisted higher-rank graph \(\Cst\)\nb-algebras~\cite{Kumjian-Pask2000Higher-Rank-graph-cst-alg, Kumjian-Pask-Sims2015Twisted-Higher-rank-graph-alg}, which can be realized as Fell line bundles over path groupoids. We show that our construction yields a natural coaction on twisted higher-rank graph algebras and identify the corresponding coaction crossed product with the \(\Cst\)\nb-algebra of the twisted skew-product higher-rank graph (Proposition~\ref{prop-identfy-skew-prod-coaction}). This extends a result of Kaliszewski--Quigg--Raeburn~\cite[Theorem 2.4]{Kaliszewski-Quigg-Raeburn-2001-Skew-prod-cross-prod-by-coaction} to the setting of twisted higher-rank graphs.

\smallskip
\paragraph{\itshape Organization of the paper.}
Section~\ref{sec:Prelim} recalls background on Fell bundles, their \(\Cst\)\nb-algebras, and coactions of discrete groups. In Section~\ref{sec-Fell-bund-as-grad}, we prove the inclusion theorem and establish the topological grading. Section~\ref{sec-coaction} contains the construction of the coaction and the proof of the coaction crossed product identification. Section~\ref{sec-application} is devoted to applications to twisted higher-rank graph \(\Cst\)\nb-algebras.

\section{Preliminaries}\label{sec:Prelim}

In this section, we recall the necessary background on groupoids, Fell bundles over groupoids, their \(\Cst\)\nb-algebras, and crossed products by coactions of discrete groups. In this article, we shall consider Fell bundles that are \emph{saturated} and \emph{separable}.

Throughout this article, \(\Hilm\) denotes an infinite dimensional separable Hilbert space, \(\Bound(\Hilm)\) the algebra of bounded linear operators on \(\Hilm\), and, for a \(\Cst\)\nb-algebra \(A\), \(M(A)\) its multiplier algebra.

\subsection{Groupoids and their \(\Cst\)-algebras} 
A \emph{groupoid} is a small category \(G\) in which every element is invertible. The collection of units of \(G\) is denoted by \(\base\), while the set of all composable pairs is denoted by \(G^{(2)}\). The multiplication is given by a map \(G^{(2)} \to G\). Additionally, there are two maps associated with \(G\): the range map \(r\) and the source map \(s\), defined by \(r,s\colon G \to \base\) by \(r(g) = g g^{-1}\) and \(s(g) = g^{-1}g\). 
A groupoid is called \emph{locally compact Hausdorff} if \(G\) equipped with a locally compact Hausdorff topology such that the structure maps are continuous. A \emph{Haar system} on \(G\) is a fully supported left invariant continuous family of measures \(\lambda=\{\lambda^u\}_{u\in \base}\) along the range map \(r\) (see~\cite[Section I.2]{Renault1980A-gpd-appr-to-cst-alg}). The groupoid \(G\) is called \emph{\'etale} if the range (or equivalently, source) map is a local homeomorphism.

A continuous \(\mathbb{T}\)-valued \(2\)\nb-cocycle \(\sigma\) on \(G\) is a continuous map \(\sigma\colon G^{(2)} \to \mathbb{T}\) satisfying \(\sigma(r(\gamma), \gamma) = \sigma(\gamma, s(\gamma)) =1\) and \(
 \sigma(\alpha, \beta)\sigma(\alpha\beta, \gamma) = \sigma(\alpha, \beta\gamma)\sigma(\beta, \gamma) \)
for all composable triples \((\alpha, \beta, \gamma)\). The set of all continuous \(2\)\nb-cocycle is denoted by \(Z^2(G, \mathbb{T})\). Let \(\sigma\) be a continuous \(2\)\nb-cocycle on a locally compact Hausdorff groupoid~\(G\) equipped with a Haar system \(\lambda\). Then the space of all compactly supported continuous function on \(G\) is denoted by \(\Cc(G, \sigma)\) and forms a \(^*\)\nb-algebra with respect to the operations 
\[
 \xi*\eta(g) = \int_{G} \xi(h)\eta(h^{-1}g) \sigma(h,h^{-1}g) \dd \lambda^{r(g)}(h)\quad \textup{and} \quad \xi^*(g) = \overline{\sigma(g,g^{-1})\xi(g)}
\] 
for \(\xi, \eta\in \Cc(G, \sigma)\). The \(I\)\nb-norm on \(\Cc(G,\sigma)\) is given by 
\[
\norm{\xi}_I = \max \Big\{\sup_{u\in \base}\int_{r^{-1}(u)}|\xi(g)| \dd\lambda^u(g), \sup_{u\in \base}\int_{s^{-1}(u)}|\xi(g)| \dd\lambda^u(g)\Big\}.
\]
A representation \(\pi\colon \Cc(G, \sigma) \to \Bound(\Hilm)\) is called \(I\)\nb-bounded if \(\norm{\pi(\xi)}\leq \norm{\xi}_I\) for all \(\xi \in \Cc(G,\sigma)\).
The \emph{twisted groupoid \(\Cst\)-algebra} \(\Cst(G,\sigma)\) is the completion of \(\Cc(G, \sigma)\) with respect to the \emph{universal norm}
\[
\norm{\xi}=\sup\big\{\norm{\pi(\xi)} : \pi\textup{ is a \(I\)-norm bounded representation of }\Cc(G, \sigma)\big\}.
\]

\subsection{Fell bundle \(\Cst\)-algebras}
Our reference for Fell bundles and upper semicontinuous bundle of Banach spaces are~\cite{Fell1988Representation-of-Star-Alg-Banach-bundles},~\cite{Muhly-Williams2008Equivalence-and-disinte-thm-Fell-bundle} and~\cite{Kumjian1998Fell-bundles-over-gpd}.		
An upper semicontinuous Banach bundle over a topological space \(X\) is a topological space \(\A\) with a continuous bundle map \(p\colon \A\to X\) such that each fibre \(\A_x\defeq p^{-1}(x)\) is a Banach space and the map \(a\mapsto \norm{a} \) from \(\A \to \R\) is an upper semicontinuous function. Let \(G\) be a locally compact groupoid and \( p\colon \mathcal{A} \to G \) an upper semicontinuous Banach bundle over \( G \). We define 
\[
\mathcal{A}^{(2)} = \big\{ (a, b) \in \mathcal{A} \times \mathcal{A} : (s(p(a)), r(p(b))) \in G^{(2)} \big\}.
\]

\begin{definition}[Fell bundle~\cite{Muhly-Williams2008Equivalence-and-disinte-thm-Fell-bundle}]
	\label{def:fell-bundle}
	A \emph{Fell bundle} over a groupoid~\(G\) is an upper
	semicontinuous bundle of Banach spaces \(p\colon \A\to G\) equipped
	with a continuous `multiplication' map \( \A^{(2)}\to \A, (a,b) \mapsto ab\)
	and an `involution' map \(\A\to \A, a\mapsto a^*\) for \(a\in \A\)
	which satisfy the following axioms:
	\begin{enumerate}[leftmargin = *]
		\item\label{item:Fell-1} \(p(ab) =p(a)p(b)\) for all
		\((a,b) \in \A^{(2)}\);
		\item\label{item:Fell-2} \(p(a^*)=p(a)^{-1}\) for all \(a\in \A\);
		
		\item \label{item:Fell-4} for each \(u\in \base\), \(\A_u\) is a \(\Cst\)\nb-algebra with respect to the inherited
		multiplication and involution on \(\A_u\);
		
		\item \label{item:Fell-5} for \(\gamma\in G\), \(\A_\gamma\) is a \(\A_{r(\gamma)}\)-\(\A_{s(\gamma)}\)-imprimitivity bimodule equipped
		with the inherited actions and inner products given by
		\[
		{}_{\A_{r(\gamma)}}\inpro{a}{b}\defeq ab^* \quad \textup{and} \quad \inpro{a}{b}_{\A_{s(\gamma)}} \defeq a^*b.
		\]

	\end{enumerate}
\end{definition}

For the rest of this paper, we fix a Fell bundle \(p\colon \A \to G\) over a locally compact, Hausdorff, second countable  groupoid~\(G\) equipped with a Haar system \(\lambda\). Let \(\Cc(G;\A)\) be the space of all compactly supported continuous sections of the bundle \(p\colon \A \to G\). Define \(\Cc(G;\A)\) as a \(^*\)\nb-algebra using the following convolution and involution operations:
\[
 \xi*\eta(g) = \int_{G} \xi(h)\eta(h^{-1}g) \dd \lambda^{r(g)}(h) \quad \text{and} \quad \xi^*(g) = \xi(g^{-1})^*
\]
where \(\xi, \eta \in \Contc(G;\A)\).
The \(I\)\nb-norm on \(\Cc(G;\A)\) is given by 
\[
\norm{\xi}_I = \max \Big\{\sup_{u\in \base}\int_{r^{-1}(u)}\|\xi(g) \| \dd\lambda^u(g), \sup_{u\in \base}\int_{s^{-1}(u)}\|\xi(g)\| \dd\lambda^u(g)\Big\}.
\]

The \emph{Fell bundle \(\Cst\)-algebra}, denoted by \(\Cst(G,\A)\), is the completion of \(\Cc(G;\A)\) with respect to the \emph{universal norm}
\[
\norm{\xi}=\sup\{\norm{\pi(\xi)} : \pi\textup{ is a \(I\)-norm bounded representation of }\Cc(G;\A)\}.
\]
\begin{obs}\label{obs-I-norm}
Using a similar argument as~\cite[Ex 1.4.4 on P.355]{Williams2019A-toolkit-for-gpd-alg} one concludes that \(I\)\nb-norm is sub-multiplicative for Fell bundles over groupoids.
\end{obs}
\begin{remark}\label{rmk-pull-back}
Let \(G, H\) be two groupoids and \(p\colon \A \to H\) a Fell bundle over \(H\). Suppose \(\phi\colon G \to H\) is a groupoid homomorphism. Then the \emph{pull-back} 
\[
 \phi^*\A =\{(g,a)\in G\times \A : \phi(g) =p(a) \}
\]
is a Fell bundle over \(G\) with respect to the operations \((g,a)(h,b) = (gh,ab)\) and \((g,a)^* =(g^{-1},a^*)\) (see~\cite[Lemma 1.1]{Kaliszewski-Muhly-Quigg-Williams-2010-Coactions-and-Fell-bundles}). Note that the fibre at \(g\in G\) can be identified with \(\A_{\phi(g)}\).
\end{remark}

\subsection{Coaction crossed products}\label{sec-cocation-dis}
We refer the reader to~\cite{Echterhoff-Kaliszewski-Quigg-Raeburn-book-categorical-appro-to-imprimitivity-thm},~\cite{ Quigg-1996-Discrete-C-st-coaction-and-C-st-bundles},~\cite{ Echterhoff-Quigg-Maximal-cocation} for a details treatment on coactions of groups.
Let \(\Gamma\) be a discrete group.
The comultiplication of \(\Cst(\Gamma)\) is the homomorphism
\[
\delta_{\Gamma}\colon \Cst(\Gamma) \to \Cst(\Gamma)\otimes \Cst(\Gamma)
\]  
given by the integrated form of the unitary homomorphism from \(\Gamma \to \Cst(\Gamma)\otimes \Cst(\Gamma)\) by \(\gamma\mapsto \gamma\otimes \gamma\). 
	Let \(A\) be a \(\Cst\)-algebra. A coaction of \(\Gamma\) on \(A\) is an injective homomorphism \(\delta \colon A \to M(A\otimes\Cst(\Gamma))\) satisfying coaction identity 
	\[
	(\delta \otimes \mathrm{id}) \circ \delta = (\mathrm{id}\otimes \delta_{\Gamma}) \circ \delta
	\] 
	together with the nondegeneracy condition 
	\[
	\overline{\textup{span}}\{\delta(A)\cdot (1\otimes \Cst(\Gamma))\} = A\otimes \Cst(\Gamma).
	\]
Note that the above nondegeneracy condition implies that \(\delta\) is nondegenerate as a map from \(A\to M(A\otimes \Cst(\Gamma))\).

A covariant representation of a coaction \((A, \Gamma, \delta)\) in \(M(B)\) is a pair \((\pi, \mu)\) consisting of nondegenerate homomorphisms \(\pi\colon A\to M(B)\) and \(\mu\colon \Gamma\to M(B)\) satisfies the covariance condition, that is, the diagram is commute
\[
\begin{tikzcd}[
	column sep=3cm,
	row sep=normal
	]
	A \arrow[r, "\delta"] \arrow[d, "\pi"'] 
	& M\bigl(A \otimes \Cst(\Gamma)\bigr) \arrow[d, "\pi \otimes \mathrm{id}"] \\
	M(B) \arrow[r, "\mathrm{Ad}(\mu \otimes \mathrm{id})(w_\Gamma) \circ (\cdot \otimes 1)"'] 
	& M\bigl(B \otimes \Cst(\Gamma)\bigr).
\end{tikzcd}
\]

A coaction crossed product \(A\rtimes_{\delta}\Gamma\) is the universal \(\Cst\)-algebra generated by the \emph{universal covariant representation} \((j_A, j_{\Gamma})\) of \((A, \Gamma,\delta)\), that is, for any covariant representation \((\pi, \mu)\) on \(M(B)\) there is a unique nondegenerate homomorphism \(\pi\times \mu \colon A\rtimes_{\delta} \Gamma \to M(B)\) such that \((\pi\times \mu)\circ j_A= \pi\) and \((\pi\times \mu)\circ j_{\Gamma}= \mu\). And 
\[
A\rtimes_{\delta}\Gamma = \overline{\textup{span}}\{j_A(a)j_{\Gamma}(1_h): a\in A, h\in \Gamma\}
\]
where \(1_h\) is the characteristic function of \(\{h\}\).
\section{Fell bundles \(\Cst\)-algebras as topological graded algebras}\label{sec-Fell-bund-as-grad}
 
In this section, we prove under certain  hypothesis on the underlying  groupoid of a Fell bundle, the \(\Cst\)\nb-algebra of the Fell bundle can be realized as a topologically graded \(\Cst\)\nb-algebra. We shall use this graded structure in Section~\ref{sec-coaction} to obtain a coaction on the Fell bundle algebra \(\Cst(G;\A)\). Most of the techniques used in the current section are motivated by~\cite{Armstrong-Clark-an-Huef-Inclusion-graded-gpd}.

Let \(G\) be a locally compact Hausdorff second countable groupoid with a Haar system \(\lambda\) and let \(\Gamma\) be a discrete group with a continuous \(1\)\nb-cocycle \(c\colon G \to \Gamma\). For \(\gamma \in \Gamma\), set \(G_{\gamma} \defeq c^{-1}(\gamma)\). Then \(G_{e}\) is a clopen subgroupoid of \(G\), where \(e\) is the identity element of \(\Gamma\). We equip \(G_e\) with the Haar system obtained by restricting \(\lambda\); that is, \(\{\lambda^{u}|_{G_e} : u\in \base\}\).
For simplicity, we continue to denote the same notation \(\lambda^u\) for \(\lambda^{u}|_{G_{e}}\).

 Consider a Fell bundle \(p\colon \A\to G\) over~\(G\). Let \(\xi \in \Cc(G;\A)\), then \(\supp(\xi) \subseteq \bigcup_{\gamma \in \Gamma} c^{-1}(\gamma)\). Since \(\supp(\xi)\) is a compact subset of \(G\), there exists a finite subset \(F\subseteq \Gamma\) such that \(\supp(\xi) \subseteq \bigcup_{\gamma \in F} c^{-1}(\gamma)\). We define \(\xi_{\gamma}\defeq \xi|_{c^{-1}(\gamma)}\) for \(\gamma\in \Gamma\). Then \(\xi = \sum_{\gamma\in \Gamma} \xi_{\gamma}\). Since \(\xi_{\gamma} =0\) if \(\gamma \notin F\), the sum survives only for a finitely many \(\gamma\). As \(c^{-1}(\gamma)\) is a clopen subset of \(G\), for \(\gamma\in \Gamma\), we can extend \(\xi_{\gamma}\) on \(\Cc(G;\A)\) by declaring zero outside \(G_{\gamma}\); we denote this extension by \(\widetilde{\xi_{\gamma}}\). 

For \(\xi, \eta\in \Cc(G;\A)\) and \(g\in \Cc(G_{e}; \A|_{G_e})\), we defined
\begin{align}
\label{eq:right-action}  \xi*_Rg(x) &\defeq \int_{G_e} \xi(xt)g(t^{-1}) \dd \lambda^{s(x)}(t)  \\
\label{eq:left-action}	\xi*_L\eta (y) &\defeq \xi*\eta (y) = \int_{G} \xi(x)\eta(x^{-1}y)\dd\lambda^{r(y)}(x) \\
\label{eq:inner-prod}	\inpro{\xi}{\eta} (x) &\defeq \sum_{\gamma \in \Gamma} \xi^*_{\gamma} \bullet \eta_{\gamma}(x)
\end{align}
where \(x\in G_e, y\in G\) and 
\[
\xi^*_{\gamma} \bullet \eta_{\gamma}(x) = \int_{G} \widetilde{\xi_{\gamma}}^*(z) \widetilde{\eta_{\gamma}}(z^{-1}x) \dd \lambda^{r(x)}(z).
\]
 
 We will show that Equations~\eqref{eq:right-action} and~\eqref{eq:left-action}
 make \(\Cc(G;\A)\) a bimodule over \(\Cc(G;\A)\) and
 \(\Cc(G_e;\A|_{G_e})\), respectively; and that Equation~\eqref{eq:inner-prod} endows \(\Cc(G;\A)\) with the structure of a pre-Hilbert 
 module over \(\Cc(G_e;\A|_{G_e})\), carrying a left action of \(\Cc(G;\A)\) given by Equation~\eqref{eq:left-action}. Finally, we complete the bimodule to obtain a \(\Cst\)\nb-correspondence from \(\Cst(G;\A)\to \Cst(G_e;\A|_{G_e})\). Before proving this we have the following observation, which will be useful in subsequent computations. 
 \begin{obs}\label{obs-corr-1}
 	Let \(\xi,\eta\in \Cc(G;\A)\) and \(g\in \Cc(G_e; \A|_{G_e})\). Then, we have
 	\begin{enumerate}
 		\item \((\eta*_Rg)_{\gamma} = \eta_{\gamma}*_Rg \);
 		\item  \(\xi^*_{\gamma}\bullet (\eta*_Rg)_{\gamma} = (\xi^*_\gamma\bullet\eta_{\gamma})*_R g\)
 		\end{enumerate}
 		for \(\gamma \in \Gamma\).
 \end{obs}
 \begin{proof}
 \noindent (1).	For \(x\in G_{\gamma}\), we have 
 	\[
 	 (\eta*_Rg)_{\gamma}(x) = (\eta*_Rg)|_{c^{-1}(\gamma)} (x) = \int_{G_e} \eta(xt)g(t^{-1}) \dd \lambda^{s(x)}(t).
 	\]
Since \(c\) is a \(1\)\nb-cocycle, \(xt\in c^{-1}(\gamma)\) for \(t\in G_e\). Thus the last term can be written as 
\[
 \int_{G_e} \eta|_{c^{-1}(\gamma)}(xt)g(t^{-1}) \dd \lambda^{s(x)}(t) = \eta_{\gamma}*_Rg(x).
 \]

 \noindent (2).
 Let \(x\in G_{\gamma}\). Now using Part~(1), we get 
 \begin{multline*}
  \xi^*_{\gamma}\bullet (\eta*_Rg)_{\gamma}(x) = \xi^*_{\gamma}\bullet (\eta_{\gamma}*_Rg)(x) = \int_{G} \widetilde{\xi_{\gamma}}^*(z)\widetilde{(\eta_{\gamma}*_Rg)}(z^{-1}x) \dd \lambda^{r(x)}(z) \\ 
 = \int_{G} \widetilde{\xi_{\gamma}}^*(z) \int_{G_e} \widetilde{\eta_{\gamma}}(z^{-1}xt)g(t^{-1}) \dd \lambda^{s(x)}(t) \dd\lambda^{r(x)}(z).
  \end{multline*}
Using Fubini theorem the last term can be written as 
\begin{multline*}
\int_{G_e}\int_{G} \widetilde{\xi_{\gamma}}^*(z)\widetilde{\eta_{\gamma}}(z^{-1}xt) g(t^{-1})  \dd\lambda^{r(x)}(z) \dd \lambda^{s(x)}(t) \\
= \int_{G_e} (\xi^*_{\gamma}\bullet\eta_{\gamma})(xt)g(t^{-1}) \dd \lambda^{s(x)}(t) = (\xi^*_\gamma\bullet\eta_{\gamma})*_R g(x).
\end{multline*}
 \end{proof}

\begin{lemma}\label{lem-bi-module}
For \(\xi, \eta, \zeta \in \Cc(G;\A)\) and \(g,h\in \Cc(G_e;\A|_{G_e})\), we have 
\begin{align}
	(\xi+\eta)*_Rg &= \xi*_Rg + \eta*_R g 	\label{equ-corr-1}\\
	\xi*_R(g+h) &= \xi*_Rg + \xi*_Rh 	\label{equ-corr-2}\\
	\xi*_R(g*h) &= (\xi*_Rg)*_Rh 	\label{equ-corr-3}\\
	\xi*_L(\eta +\zeta) &= \xi*_L\eta + \xi*_L\zeta 	\label{equ-corr-4}\\
	(\xi+\eta)*_L\zeta &= \xi*_L\eta + \xi*_L \zeta 	\label{equ-corr-5}\\
	(\xi*\eta)*_L\zeta &= \xi*_L(\eta*_L\zeta) 	\label{equ-corr-6}\\
	(\xi*_L\eta)*_Rg &=\xi*_L(\eta*_Rg) 	\label{equ-corr-7}\\
	\inpro{\xi}{\eta+\zeta} &= \inpro{\xi}{\eta} +\inpro{\xi}{\zeta} 	\label{equ-corr-8}\\
	\inpro{\xi}{\eta*_Rg} &=\inpro{\xi}{\eta}*g 	\label{equ-corr-9} \\
	\inpro{\xi}{\eta}^* &= \inpro{\eta}{\xi} 	\label{equ-corr-10}\\
	\inpro{\xi*_L\eta}{\zeta} &= \inpro{\eta}{\xi^**_L\zeta}. 	\label{equ-corr-11}
	\end{align}
\end{lemma}
\begin{proof}
 Equations~\eqref{equ-corr-1}, \eqref{equ-corr-2}, \eqref{equ-corr-4}, \eqref{equ-corr-5} and~\eqref{equ-corr-8} follows from the definition. Equation~\eqref{equ-corr-7} follows from the associativity of the convolution.

The proof of Equations~\eqref{equ-corr-3} and~\eqref{equ-corr-6} are similar. Therefore, we only prove Equation~\eqref{equ-corr-3}.
\begin{align*}
\xi*_R(g*h)(x) &= \int_{G_e} \xi(xt)(g*h)(t^{-1}) \dd \lambda^{s(x)}(t) \\
&= \int_{G_e} \xi(xt) \int_{G_e} g(y)h(y^{-1}t^{-1}) \dd \lambda^{s(t)}(y) \dd \lambda^{s(x)}(t) \\
&= \int_{G_e} \int_{G_e} \xi(xt) g(y)h(y^{-1}t^{-1}) \dd \lambda^{s(t)}(y) \dd \lambda^{s(x)}(t).
\end{align*}
Using the change of variables \(ty\mapsto z\), the last term above becomes
\[
 \int_{G_e} \int_{G_e} \xi(xt)g(t^{-1}z) h(z^{-1}) \dd\lambda^{s(x)}(z) \dd \lambda^{s(x)}(t).
\]
We now use Fubini theorem to rewrite the last term as 
\[
 \int_{G_e} \int_{G_e} \xi(xt)g(t^{-1}z) h(z^{-1}) \dd\lambda^{s(x)}(t) \dd \lambda^{s(x)}(z).
\]
Using the change of variables \(t\mapsto zt\), we obtain
\begin{multline*}
\int_{G_e} \int_{G_e} \xi(xzt)g(t^{-1})g(t^{-1}) h(z^{-1}) \dd \lambda^{s(z)}(t) \dd\lambda^{s(x)}(z)\\
= \int_{G_e} (\xi*_Rg) (xz)h(z^{-1}) \dd \lambda^{s(x)}(z) = (\xi*_Rg)*_Rh(x).
\end{multline*} 

\noindent For Equation~\eqref{equ-corr-9},
\begin{align*}
\inpro{\xi}{\eta*_Rg}(x) &= \sum_{\gamma\in \Gamma} \xi^*_{\gamma}\bullet (\eta*_Rg)_{\gamma}(x) = \sum_{\gamma\in \Gamma} (\xi^*_{\gamma}\bullet \eta_{\gamma})*_Rg(x)\\
&= \sum_{\gamma\in \Gamma} (\xi^*_{\gamma}\bullet \eta_{\gamma})*g(x) =  \Big(\sum_{\gamma\in \Gamma} \xi^*_{\gamma}\bullet \eta_{\gamma}\Big)*g(x) = \inpro{\xi}{\eta}*g(x).
\end{align*}
The second equality in the first line follows form Observation~\ref{obs-corr-1}(2) and second equality in the second line follows because the sum is finite.

\noindent For Equation~\eqref{equ-corr-10},
\begin{align*}
\inpro{\xi}{\eta}^*(x) &= (\inpro{\xi}{\eta}(x^{-1}))^* = \Big(\sum_{\gamma\in \Gamma}\xi^*_{\gamma}\bullet \eta_{\gamma}(x^{-1})\Big)^* \\
&= \Big(\sum_{\gamma\in \Gamma} \int_{G} \widetilde{\xi_{\gamma}}^*(z) \widetilde{\eta_{\gamma}}(z^{-1}x^{-1}) \dd\lambda^{s(x)}(z)\Big)^*\\
&= \sum_{\gamma \in \Gamma} \int_{G} \bigl(\widetilde{\eta_{\gamma}}(z^{-1}x^{-1})\bigr)^* \widetilde{\xi_{\gamma}}(z^{-1}) \dd\lambda^{s(x)}(z)  \\
&=  \sum_{\gamma \in \Gamma} \int_{G} \widetilde{\eta_{\gamma}}^*(xz) \widetilde{\xi_{\gamma}}(z^{-1}) \dd\lambda^{s(x)}(z).
\end{align*}
Using the change of variables \(z\mapsto x^{-1}z\), the last term can be written as 
\[
 \sum_{\gamma \in \Gamma} \int_{G} \widetilde{\eta_{\gamma}}^*(z) \widetilde{\xi_{\gamma}}(z^{-1}x) \dd\lambda^{r(x)}(z) =  \sum_{\gamma \in \Gamma} \eta_{\gamma}^*\bullet \xi_{\gamma}(x) = \inpro{\eta}{\xi}(x).
\]
For Equation~\eqref{equ-corr-11}, we assume \(\supp(\xi)\subseteq c^{-1}(\delta)\) for some \(\delta \in \Gamma\), 
\begin{align*}
\inpro{\xi*_L\eta}{\zeta} (x) &= \sum_{\gamma \in \Gamma} (\xi*_L\eta)_{\gamma}^*\bullet \widetilde{\zeta_{\gamma}}(x) 
= \sum_{\gamma \in \Gamma} \int_{G} \widetilde{(\xi*_L\eta)_{\delta\gamma}}^*(z) \widetilde{\zeta_{\delta\gamma}}(z^{-1}x) \dd \lambda^{r(x)}(z)\\
&= \sum_{\gamma \in \Gamma} \int_{G} \int_{G} \bigl( \widetilde{\xi_{\delta}}(y) \widetilde{\eta_{\gamma}}(y^{-1}z^{-1})\bigr)^* \dd \lambda^{s(z)}(y) \widetilde{\zeta_{\delta\gamma}}(z^{-1}x) \dd \lambda^{r(x)}(z) \\
&= \sum_{\gamma \in \Gamma} \int_{G} \int_{G}  \widetilde{\eta_{\gamma}}^*(zy) \widetilde{ \xi_{\delta}}^*(y^{-1}) \widetilde{\zeta_{\delta\gamma}}(z^{-1}x)  \dd \lambda^{s(z)}(y) \dd \lambda^{r(x)}(z)
\end{align*}
Using the change of variables \(y\mapsto z^{-1}y\), the last term can be written as 
\[
\sum_{\gamma \in \Gamma} \int_{G} \int_{G}  \widetilde{\eta_{\gamma}}^*(y)  \widetilde{\xi_{\delta}}^*(y^{-1}z) \widetilde{\zeta_{\delta\gamma}}(z^{-1}x)  \dd \lambda^{r(z)}(y) \dd \lambda^{r(x)}(z).
\]
Again, using the change of variables \(z\mapsto yz\) and Fubini theorem, the above term becomes,
\begin{align*}
	&\sum_{\gamma \in \Gamma} \int_{G} \int_{G}  \widetilde{\eta_{\gamma}}^*(y)  \widetilde{\xi_\delta}^*(z) \widetilde{\zeta_{\delta\gamma}}(z^{-1}y^{-1}x) \dd \lambda^{s(y)}(z) \dd \lambda^{r(x)}(y) \\
	&= \sum_{\gamma \in \Gamma} \int_{G}  \widetilde{\eta_{\gamma}}^*(y) (\widetilde{\xi^**_L\zeta})_{\gamma}(y^{-1}x)  \dd \lambda^{r(x)}(y) \\
	&= \sum_{\gamma \in \Gamma}  \eta^*_{\gamma}\bullet (\xi^**_L\zeta)_{\gamma} = \inpro{\eta}{\xi^**_L\zeta}.
\end{align*}
If \(\xi\in \Cc(G;\A)\), then we can cover the compact set \(\supp(\xi)\) by the open cover \(\{c^{-1}(\gamma)\}_{\gamma\in \Gamma}\) and hence admits a finite subcover. Thus, using a partition of unity argument, we conclude that Equation~\eqref{equ-corr-11} holds.
\end{proof}

\begin{lemma}\label{lem-con-ind-lim-1}
	\begin{enumerate}
		\item Suppose \((\xi_n)_{n\in \N}\) is a sequence of sections in \(\Cc(G;\A)\) that converges
	to the zero section in the inductive limit topology. Then \(\inpro{\xi_n}{\xi_n}\to 0\) in \(\Cc(G_e; \A|_{G_e})\) in the inductive limit topology, hence in \(\Cst(G_e;\A|_{G_e})\) as well.
	\item Let \(g\in \Cc(G;\A)\). If \(f_n\to f\) in \(\Cc(G;\A)\) in the inductive limit
	topology, then \(f_n*_Lg \to f*_Lg\) in \(\Cc(G;\A)\) in the inductive limit topology.
\end{enumerate}
\end{lemma}
\begin{proof}
\noindent (1). Since \(\xi_n \to 0\) in the inductive limit topology, there is a fixed compact set~\(K\) such that \(\supp(\xi_n) \subseteq K\) for all \(n\). We write \(\xi_n =\sum_{\gamma \in \Gamma} (\xi_n)_\gamma\), with \(\supp((\xi_n)_\gamma) \subseteq G_{\gamma}\).
Therefore, 
\[
\inpro{\xi_n}{\xi_n} = \sum_{\gamma \in \Gamma} (\xi_n)^*_\gamma\bullet (\xi_n)_\gamma \to 0 
\]
uniformly as \(n\to \infty\) because each \((\xi_n)^*_\gamma\bullet (\xi_n)_\gamma \to 0\) uniformly as \(n\to \infty\) and supported in the compact set \((K\cap G_\gamma)^{-1}(K\cap G_\gamma)\). Since \(\supp(\inpro{\xi_n}{\xi_n})\) is in the compact set \(\bigcup_{\textup{finite}}(K\cap G_\gamma)^{-1}(K\cap G_\gamma) \), \(\inpro{\xi_n}{\xi_n} \to 0\) in the inductive limit topology.

\noindent (2). This part follows from a similar argument as Lemma~3.19(2) of~\cite{Holkar-Hossain-2024-KMS-states}.
 \end{proof}

\begin{lemma}\label{lem-con-ind-lim-2}
	Let \(g, h\in \Cc(G;\A)\). Then the map \(f\mapsto \inpro{g}{f*_Lh}\)
	from \(\Cc(G;\A)\) to \(\Cc(G_e;\A|_{G_e})\) is continuous in the inductive limit topology.
\end{lemma}
\begin{proof} 
Let \((f_n)_{n\in \N}\) be a sequence in \(\Cc(G;\A)\) converges to \(f\) in the inductive limit topology. Let \(K\) be a fixed compact set with \(\supp(f_n)\) and \(\supp(f)\) contained in~\(K\). Then \(\supp(\inpro{g}{f_n*_Lh})\) and \(\supp(\inpro{g}{f*_Lh})\) are also contained in some fixed compact set. Using Lemma~\ref{lem-bi-module} and Cauchy--Schwarz inequality, we have 
\begin{align*}
\norm{\inpro{g}{f_n*_Lh}-\inpro{g}{f*_Lh}}^2 &= \norm{\inpro{g}{(f_n-f)*_Lh}}^2\\ 
&= \norm{\inpro{g}{(f_n-f)*_Lh}^*\inpro{g}{(f_n-f)*_Lh}} \\
&\leq \norm{\inpro{g}{g}} \norm{\inpro{(f_n-f)*_Lh}{(f_n-f)*_Lh}}.
\end{align*}
Since \(f_n\to f\) in the inductive limit topology, Lemma~\ref{lem-con-ind-lim-1}(2) ensures that \(f_n*_Lh \to f*_Lh\) in the inductive limit topology and Lemma~\ref{lem-con-ind-lim-1}(1), gives us
 \[
 \inpro{(f_n-f)*_Lh}{(f_n-f)*_Lh} \to 0
 \] 
 in the inductive limit topology. This completes the proof.
\end{proof}

\begin{lemma}\label{lem-inner-prod-positive}
	\begin{enumerate}
		\item For \(\xi\in \Cc(G;\A)\), \(\inpro{\xi}{\xi}\geq 0\) in \(\Cc(G_{e}; \A|_{G_e})\).
		\item  \(\textup{span}\{\inpro{\xi}{\eta} : \xi, \eta \in \Cc(G;\A)\}\) is dense in \(\Cst(G_{e}; \A|_{G_e})\).
		
		\item Let \(\eta, \zeta \in \Cc(G;\A)\). If \(\xi_n\to \xi\) in \(\Cc(G;\A)\) in the inductive limit topology, then \(\inpro{\xi_n*_L\eta}{\zeta} \to \inpro{\xi*_L\eta}{\zeta}\) in \(\Cc(G_{e};\A|_{G_e})\) in the inductive limit topology. 
	\end{enumerate}
\end{lemma}
 \begin{proof}
\noindent (1).  
We have \(\inpro{\xi}{\xi} = \sum_{\gamma \in F} \xi^*_{\gamma}\bullet\xi_{\gamma}\) for a finite subset \(F\subseteq \Gamma\). As sum of positive elements is positive, \(\inpro{\xi}{\xi}\) is positive in \(\Cc(G;\A)\). Since \(\inpro{\xi}{\xi} \in \Cc(G_e; \A|_{G_e})\) and \(\overline{\Cc(G_e;\A|_{G_e})}^{I} \subseteq \overline{\Cc(G;\A)}^{I}\), spectral permanence allows us to conclude that \(\inpro{\xi}{\xi}\) is positive in \(\Cc(G_e;\A|_{G_e})\).

\noindent (2).
Let \(f\in \Cc(G_e; \A|_{G_e})\). Choose a self-adjoint approximate identity of \((u_\alpha)_{\alpha}\) in \(\Cst(G_e; \A|_{G_e})\) such that \(u_\alpha \in \Cc(G_e; \A|_{G_e})\) with each \(u_\alpha\) has norm al most \(1\) (since the groupoid \(G_e\) is second countable the Fell bundle is saturated such approximate identity always exists see~\cite[Proposition 6.10]{Muhly-Williams2008Equivalence-and-disinte-thm-Fell-bundle}).
Let \(f\in \Cc(G_e; \A|_{G_e})\). Then 
\[
\inpro{u_\alpha}{f} = (u_\alpha)^*_e\bullet (f)_e= u_\alpha\bullet f = u_\alpha * f \to f. 
\]

\noindent (3). The proof follows from a similar argument as Lemma~\ref{lem-con-ind-lim-2}.
 \end{proof}
\medskip
Lemma~\ref{lem-bi-module} shows that \(\Cc(G;\A)\) is a right module over \(\Cc(G_{e}; \A|_{G_e})\), and Lemma~\ref{lem-inner-prod-positive}(2) shows that the \(\Cc(G_e;\A|_{G_e})\)-valued inner product \(\inpro{\cdot}{
\cdot}\) defined by Equation~\eqref{eq:inner-prod} is positive definite. Let \(Y\) be the Hilbert \(\Cst(G_e; \A|_{G_e})\)-module obtained by completing the pre-Hilbert \(\Cst(G_e;\A|_{G_e})\)\nb-module \(\Cc(G;\A)\). Lemma~\ref{lem-inner-prod-positive}(2) implies that \(Y\) is a full Hilbert
module.
\begin{lemma}\label{lem-left-action-non-deg}
The \(\textup{span}\{\xi*_L\eta: \xi, \eta\in \Cc(G,\A)\}\) is dense in the right Hilbert-\(\Cst(G_e; \A|_{G_e})\) module \(Y\).
\end{lemma}
\begin{proof}
Let \((u_\alpha)_{\alpha}\) be a self-adjoint approximate identity of \(\Cst(G_e;\A|_{G_e})\) such that \(u_\alpha \in \Cc(G_e;\A|_{G_e})\) and \(\norm{u_\alpha} \leq 1\). Let \(f\in \Cc(G;\A)\) with \(f =\sum_{\gamma \in F} f_\gamma\), where \(F\) is a finite subset of \(\Gamma\).
Now, we have 
\begin{align*}
\norm{f*_L\widetilde{u_\alpha} - f}^2_{\Cst(G_e; \A|_{G_e})}  &= \norm{(f*_L\widetilde{u_\alpha} - f)^*(f*_L\widetilde{u_\alpha} - f)} \\
&= \bigr\|\sum_{\gamma \in F}(f_\gamma *u_\alpha - f_\gamma)^*(f_\gamma*u_\alpha - f_{\gamma}) \bigr \|\\
&= \bigr\| \sum_{\gamma \in F} u^*_\alpha f^*_{\gamma} f_{\gamma} u_{\alpha} - u^*_{\alpha} f^*_{\gamma} f_{\gamma} - f^*_{\gamma} f_{\gamma} u_{\alpha} -f^*_{\gamma} f_{\gamma} \bigr\| \\
&\leq \sum_{\gamma \in F} \big(\norm{u_\alpha} \norm{f^*_{\gamma} f_{\gamma} u_{\alpha}- f^*_{\gamma} f_{\gamma}} + \norm{ f^*_{\gamma} f_{\gamma} - f^*_{\gamma} f_{\gamma} u_{\alpha}} \big)\\
&\leq 2\sum_{\gamma \in F} \norm{f^*_{\gamma} f_{\gamma} u_{\alpha}- f^*_{\gamma} f_{\gamma}}.
\end{align*}
The last term goes to zero as \(\alpha \to \infty\).
\end{proof}

We now show that the left action \(*_L\) defined by Equation~\eqref{eq:left-action}
extends to a nondegenerate
representation of \(\Cst(G;\A)\) on \(Y\) by adjointable operators. Equation~\eqref{equ-corr-11} of Lemma~\ref{lem-bi-module} ensures that \(\Cc(G;\A)\) acts on itself by adjointable operators.
To show that the left action is bounded, let \(\phi\) be a state on \(\Cst(G_e; \A|_{G_e})\). Consider the Hilbert space completion of the inner product space \((Y, \inpro{\cdot}{\cdot}_{\phi})\), where \(\inpro{f}{g}_{\phi} \defeq \phi(\inpro{f}{g}) \) for \(f,g\in Y\). We denote this Hilbert space completion by~\(Y_{\phi}\).
Define the vector subspace \(W\subseteq Y_{\phi}\) generated by \(\{\xi*_L\eta : \xi, \eta \in \Cc(G;\A)\}\). Lemma~\ref{lem-left-action-non-deg} shows that \(W\) is dense subspace of \(Y\). Define a representation \(\pi \) of \(\Cc(G;\A)\) on \(W\) by 
\[
  \pi(f)(g) = f*_Lg
\]
for \(f,g \in \Cc(G;\A)\). We now show that \(\pi\) satisfies the three conditions required to be a pre-representation (see Definition~4.1 of~\cite{Muhly-Williams2008Equivalence-and-disinte-thm-Fell-bundle}) of \(\Cc(G;\A)\) on \(W\).
\begin{enumerate}
	\item \(\pi\) is nondegenerate follows from Lemma~\ref{lem-left-action-non-deg}. 
	\item Lemma~\ref{lem-con-ind-lim-2} ensures that, the map \(f \mapsto \inpro{\xi}{\pi(f)\eta}_{\phi}\) is continuous in the inductive limit topology for \(\xi, \eta\) and \(f \in\Cc(G;\A)\).
	\item By Lemma~\ref{lem-bi-module}, we have \(\inpro{f*_L\xi}{\eta}_{\phi} = \inpro{\xi}{f^**_L\eta}_{\phi}\) for \(\xi, \eta, f\in \Cc(G;\A)\).
\end{enumerate}
Therefore, using the disintegration theorem for Fell bundles~\cite[Theorem 4.13]{Muhly-Williams2008Equivalence-and-disinte-thm-Fell-bundle} we extends \(\pi\) to be a
nondegenerate (bounded) representation of \(\Cst(G;\A)\) on \(Y_{\phi}\). Thus,
\[
\phi(\inpro{f*_L\xi}{f*_L\xi}) \leq \norm{f}^2_{\Cst(G;\A)} \phi(\inpro{\xi}{\xi})
\]
for all \(\xi, f\in \Cc(G;\A)\). Since \(\phi\) was an
 arbitrary state on \(\Cst(G_e; \A|_{G_e})\), we have 
 \[
 \inpro{f*_L\xi}{f*_L\xi} \leq \norm{f}^2_{\Cst(G;\A)} \inpro{\xi}{\xi}.
 \]
 Hence, the left action \(*_L\) of \(\Cc(G; \A)\) on itself
 extends to a
nondegenerate representation of \(\Cst(G;\A)\)
 on \(Y\).
 
 Lemmas~\ref{lem-bi-module},~\ref{lem-inner-prod-positive} and the above discussion gives us the following result.
 \begin{theorem}\label{thm-C-st-corr}
 	Let \(p\colon \A\to G\) be a Fell bundle over a locally compact Hausdorff second countable groupoid equipped with a Haar system. Let \(\Gamma\) be a discrete group and \(c\colon G\to \Gamma\) a continuous \(1\)-cocycle. Then the \(\Cc(G;\A)\)-\(\Cc(G_e;\A|_{G_e})\)-bimodule \(\Cc(G;\A)\) completes to a \(\Cst\)-correspondence 
 	\[
 	Y\colon \Cst(G;\A) \to \Cst(G_e;\A|_{G_e}).
 	\] 
 \end{theorem}
 \begin{remark}\label{remk-notation-L}
  We denote the left action of \(\Cst(G;\A)\) on \(Y\) in Theorem~\ref{thm-C-st-corr} by \(L\), that is, \(L\colon \Cst(G;\A) \to \Bound(Y)\) is a homomorphism, where \(\Bound(Y)\) denotes the set of all adjointable operators on \(Y\).
  \end{remark}
  
  We denote 
  \[
  \B_e =\overline{\{f\in \Cc(G;\A) : \supp(f) \subset G_{e}\}} \subset \Cst(G;\A)
  \]
  where \(e\) is the identity element of \(\Gamma\) and the above closure is taken with respect to the \(\Cst\)\nb-norm of \(\Cst(G;\A)\). We now prove the main theorem of this section.
  
  \begin{theorem}\label{thm-isometric-isomorphism}
  	Let \(L\colon \Cst(G;\A) \to \Bound(Y)\) be the homomorphism of Remark~\ref{remk-notation-L}. Then \(L|_{B_e}\colon B_e\to \Bound(Y)\) is injective and the inclusion map \(\iota\colon  \Cc(G_e; \A|_{G_e}) \to \Cc(G;\A)\) given by 
  	\[
  	\iota(f)(\gamma) = \begin{cases}
  		f(\gamma) & \textup{if } \gamma \in G_e,\\
  		0 & \textup{ otherwise }
  	\end{cases}
  	\]
  	extends to an isomorphism from \(\Cst(G_e;\A|_{G_e}) \to \B_e\).
  \end{theorem}
 \begin{proof}
 	Note that the map  \(\iota\colon  \Cc(G_e; \A|_{G_e}) \to \Cc(G;\A)\) is a homomorphism of \(^*\)\nb-algebras and it preserve the \(I\)\nb-norm. Therefore, \(\iota\) can be extended to a norm decreasing \(^*\)\nb-homomorphism \(\bar{\iota}\colon \Cst(G_e;\A|_{G_e}) \to \Cst(G;\A)\). Thus, for \(f\in \Cc(G_e;\A|_{G_e})\), we have \(\norm{\iota(f)} \leq \norm{f}\). Now 
 	\[
 	 \norm{\iota(f)}^2_{\Cst(G_e;\A|_{G_e})} = \norm{\inpro{\iota(f)}{\iota(f)}} =\norm{\iota(f)^*_e\iota(f)_e} = \norm{f^*f} =\norm{f}^2
 	\]
 	for \(f\in \Cc(G_e;\A|_{G_e})\). The second equality above follows because the support of \(\iota(f)\) is contained in~\(G_e\). 
 	If \(f\) is a nonzero section in \(\Cc(G_e;\A|_{G_e})\), then
 	\begin{align*}
 		\norm{L(\iota(f))} &=\sup\big\{\norm{\iota(f)g}_{\Cst(G_e;\A|_{G_e})} : g\in Y \textup{ with } \norm{g}\leq 1\big \}\\
 		&\geq \norm{\iota(f)\iota(f)^*}_{\Cst(G_e;\A|_{G_e})}/\norm{f} = \norm{ff^*}/\norm{f} =\norm{f}.
 	\end{align*}
 	Since \(L\) is norm-decreasing (by Theorem~\ref{thm-C-st-corr}), we have 
 \[
 \norm{f}\leq \norm{L(\iota(f))} \leq \norm{\iota(f)} \leq \norm{f}.
 \]
 Thus, the above line ensures that \(\iota\) can be extended to an isometric map \(\bar{\iota}\) and the restriction of \(L\) on \(\B_e\) is an isometry.
  \end{proof}
  
  Our next goal is to show that the Fell bundle \(\Cst\)\nb-algebra \(\Cst(G;\A)\) is a topological grading in the sense of Exel~\cite[Definition 19.2]{Exel2017Book-Partial-action-Fell-bundle}. 
 We now recall the definition of topological grading from~\cite{Exel2017Book-Partial-action-Fell-bundle}. A \(\Cst\)\nb-algebra \(A\) is called \emph{topological grading} over a discrete group~\(\Gamma\) if there exists a collection of closed subspace \(\{A_{\gamma} : \gamma \in \Gamma\}\) of \(A\) such that 
 \begin{enumerate}
 	\item \(A_\gamma A_{\eta} \subset A_{\gamma \eta}\) for \(\gamma, \eta\in \Gamma\);
 	\item \(A_{\gamma}^* = A_{\gamma^{-1}}\) for all \(\gamma\in \Gamma\);
 	\item the linear span of \(\{A_{\gamma} : \gamma \in \Gamma\}\) is dense in \(A\);
 	\item \label{cond:top-grading-4} there is a bounded linear map \(E \colon A \to A\) such that \(E = \textup{id}\) on \(A_{e}\) and \(E = 0\) on \(A_{\gamma}\) for \(\gamma\neq e\).
 \end{enumerate}
  \noindent If \(A\) is a topological grading over a discrete group, then Exel proved~\cite[Theorem~19.1]{Exel2017Book-Partial-action-Fell-bundle} that the map in Condition~\ref{cond:top-grading-4} of the above definition is a \emph{conditional expectation} from \(A\) onto \(A_e\) and the collection \(\{A_{\gamma} : \gamma \in \Gamma\}\) is linearly independent.

\begin{proposition}\label{prop-linear-cont}
	Let \(p\colon \A\to G\) be a Fell bundle over a locally compact Hausdorff second countable groupoid equipped with a Haar system. If \(c\colon G\to \Gamma\) is a continuous \(1\)\nb-cocycle, then there exists a linear contraction \(E\colon \Cst(G;\A) \to \Cst(G_{e}; \A|_{G_e})\) given by \(E(f) = f|_{G_e}\) for \(f\in \Cc(G;\A)\).
\end{proposition} 
 \begin{proof}
 	For \(f\in \Cc(G;\A)\), we have 
 	\begin{equation}\label{liner-cont-eq-1}
  \norm{f}^2_{\Cst(G_e;\A|_{G_e})} = \bigr \| \sum_{\gamma \in \Gamma} f^*_{\gamma} f_{\gamma} \bigr\| \geq \norm{f^*_ef_e} = \norm{f|_{G_e}}^2,
 \end{equation}
 	where the inequality of the above line follows from \(\sum_{\gamma \in \Gamma}f^*_{\gamma}f_{\gamma} \geq f^*_ef_e \geq 0\). Choose a self-adjoint approximation identity \((u_\alpha)_{\alpha}\) for \(\Cst(G_e;\A|_{G_e})\) in \(\Cc(G_e;\A|_{G_e})\) (such approximation identity exists by~\cite[Proposition 6.10]{Muhly-Williams2008Equivalence-and-disinte-thm-Fell-bundle}). Since \(\norm{\iota(u_{\alpha})} \leq 1\), we have 
 	\begin{multline*}
 	\norm{L(f)}^2 \geq \norm{L(f)(\iota(u_{\alpha}))}_{\Cst(G_e;\A|_{G_e})}^2  = \norm{f*_L\iota(u_{\alpha})}_{\Cst(G_e;\A|_{G_e})}^2\\
 	 = \norm{ u^*_{\alpha}\inpro{f}{f} u_{\alpha}} \to \norm{f}^2_{\Cst(G_e;\A|_{G_e})} \quad \textup{as } \alpha \to \infty. 
 	\end{multline*}
 	Equation~\eqref{liner-cont-eq-1}, Theorem~\ref{thm-C-st-corr} and the above computation gives us 
 	\[
 	 \norm{f|_{G_e}} \leq \norm{f}_{\Cst(G_e;\A|_{G_e})} \leq \norm{L(f)}\leq \norm{f}.
 	\]
 	Therefore, the restriction extends to a bounded linear contraction \(E\).
 \end{proof}

 \begin{proposition}\label{prpo-graded-algebra-st}
 	Let \(p\colon \A\to G\) be a Fell bundle over a locally compact Hausdorff second countable groupoid equipped with a Haar system. Let \(\Gamma\) be a discrete group with a continuous \(1\)\nb-cocycle \(c\colon G\to \Gamma\). For \(\gamma \in \Gamma\), set 
 	\[
 	 \B_{\gamma}= \overline{\{f\in \Cc(G;\A) : \supp(f) \subseteq G_{\gamma}\}} \subseteq \Cst(G;\A),
 	\]
 	where the closure is taken with respect to the \(\Cst\)\nb-norm.
 	Then \(\Cst(G;\A)\) is a topologically graded \(\Cst\)\nb-algebra. The grading subspaces are \(\B= \{\B_{\gamma}\}_{\gamma \in \Gamma}\) and the conditional expectation \(E\colon \Cst(G;\A) \to \Cst(G;\A)\) induced by the restriction of sections to \(G_e\).
 \end{proposition}
 \begin{proof}
 	Let \(\gamma, \eta \in \Gamma\) and let \(f,g\in \Cc(G;\A)\) with \(\supp(f)\subseteq G_\gamma\) and \(\supp(g) \subseteq G_{\eta}\). Since \(c\) is a \(1\)\nb-cocycle, it follows that \(\supp(f)\supp(g) \subseteq G_{\gamma} G_{\eta} \subseteq G_{\gamma\eta}\). Thus, \(\B_{\gamma}\B_{\eta} \subseteq \B_{\gamma \eta}\). Moreover, since \(\supp(f^*)\subseteq G_{\gamma^{-1}}\), we have \(\B^*_{\gamma} = \B_{\gamma^{-1}}\). 
 	
 	Let \(f\in \Cc(G;\A)\). We write \(f=\sum_{\gamma \in F}f_{\gamma}\) for some finite subset \(F\) of \(\Gamma\). Now we can represent \(f\) as \(f= \sum_{\gamma \in F}f_{\gamma} =\sum_{\gamma \in F}\iota_{\gamma}(f_{\gamma})\), where \(\iota_{\gamma}\colon \Cc(G_{\gamma};\A|_{G_\gamma}) \hookrightarrow \Cc(G;\A)\) is the inclusion. Therefore, \(\Cc(G;\A)\) is contained in \(\textup{span}\{\B_{\gamma} : \gamma \in \Gamma\}\), and hence \(\textup{span}\{\B_{\gamma} : \gamma \in \Gamma\}\) is dense in \(\Cst(G;\A)\).
 	
 	 Recall from Proposition~\ref{prop-linear-cont}, that the map \(E\colon \Cst(G;\A) \to \Cst(G_e;\A|_{G_e})\) defined by \(E(f) =f|_{G_e}\) is a linear contraction.  Composing the isometric isomorphism \(\overline{\iota}\colon \Cst(G_e;\A|_{G_e}) \to \Cst(G;\A)\) obtained in Theorem~\ref{thm-isometric-isomorphism} with \(E\), we obtain a contraction \(\widetilde{E} = \overline{\iota}\circ E\colon \Cst(G;\A) \to \Cst(G;\A)\). The map \(\widetilde{E}\) satisfies \(\widetilde{E} =0\) on \(\B_{\gamma}\) for \(\gamma \neq e\) and \(\widetilde{E} =\id\) on \(\B_e\). Thus, \(\widetilde{E}\) is a conditional expectation, and hence \(\Cst(G;\A)\) is a topologically graded \(\Cst\)\nb-algebra. 
 \end{proof}

 \begin{theorem}\label{thm-Fell-bund-as-graded-alg}
 	Let \(p\colon \A\to G\) be a Fell bundle over a locally compact Hausdorff second countable groupoid equipped with a Haar system. Let \(\Gamma\) be a discrete group with a continuous \(1\)\nb-cocycle \(c\colon G\to \Gamma\). Consider the group Fell bundle \(\B =\{\B_{\gamma}\}_{\gamma \in \Gamma}\) as described above. 
 	Then the group Fell bundle algebra \(\Cst(\Gamma; \B)\) is isomorphic to \(\Cst(G;\A)\).
 \end{theorem}
 \begin{proof}
 	Define a representation (see~\cite[Definition 16.20]{Exel2017Book-Partial-action-Fell-bundle}) \(\{\id_{\gamma}\}_{\gamma \in \Gamma}\) of \(\B\) on \(\Cst(G;\A)\) given by the identity maps \(\id_{\gamma} \colon \B_{\gamma} \to \Cst(G;\A)\). By Proposition~\ref{prpo-graded-algebra-st}, the algebra \(\Cst(G;\A)\) generated by \(\textup{span}\{\id_{\gamma}(f) : f\in \B_{\gamma}, \gamma \in \Gamma\}\). To show \(\Cst(G;\A)\) satisfies the universal property of \(\Cst(\Gamma; \B)\), let  \(\{\pi_{\gamma}\}_{\gamma \in \Gamma}\) be a representation of the group Fell bundle \(\B\) on a \(\Cst\)\nb-algebra \(M\). We aim to construct a representation \(\pi\colon \Cst(G;\A) \to M\) such that \(\pi\circ \id_{\gamma} = \pi_{\gamma}\) for \(\gamma \in \Gamma\). Define \(\pi\colon \Cc(G;\A) \to M \) by 
 		\[
 		 \pi(f) =\sum_{\gamma \in \Gamma}\pi_{\gamma}(f_{\gamma}) 
 		\]
 		for \(f = \sum_{\gamma \in \Gamma}f_{\gamma} \in \Cc(G;\A)\). By definition, \(\pi\) is multiplicative and preserves the involution. We now show that \(\pi\) is continuous in the inductive limit topology, and hence extends to a \(^*\)\nb-homomorphism from \(\Cst(G;\A) \to M\). Since we can identify \(\Cst(G_e;\A|_{G_e})\) with \(\B_e\) by the isometric isomorphism \(\overline{\iota}\colon \Cst(G_e;\A|_{G_e}) \to \B_e\) of Theorem~\ref{thm-isometric-isomorphism}, we have a \(^*\)\nb-homomorphism \(\pi_e\colon \Cst(G_e; \A|_{G_e})\to M\). Since \(\Cst(G_e; \A|_{G_e})\) is a Fell bundle algebra, \(\pi_e\) is \(I\)\nb-norm bounded on \(\Cc(G_e; \A|_{G_e})\) by~\cite[Theorem 4.13]{Muhly-Williams2008Equivalence-and-disinte-thm-Fell-bundle}. For \(f\in \Cc(G;\A)\cap \B_{\gamma}\), we have 
 		\[
 		 \norm{\pi_{\gamma}(f)}^2= \norm{\pi_{\gamma}(f)^*\pi_{\gamma}(f)} =\norm{\pi_e({f^*f})} \leq \norm{f^*f}_I\leq \norm{f^*}_I\norm{f}=\norm{f}^2_I
 		\]
 		The second inequality of the above line follows from Observation~\ref{obs-I-norm}. Since \(\textup{span}\{\B_{\gamma} :\gamma \in \Gamma\}\) is dense in \(\Cc(G;\A)\), \(\pi\) is \(I\)\nb-norm bounded on \(\Cc(G;\A)\). By the disintegration theorem (\cite[Theorem 4.13 ]{Muhly-Williams2008Equivalence-and-disinte-thm-Fell-bundle}), \(\pi\) is continuous in the inductive limit topology. Since \(\pi\circ \id_{\gamma} =\pi_{\gamma}\) holds on \(\Cc(G;\A)\), the equality extends to \(\Cst(G;\A)\). Therefore, \(\Cst(G;\A)\) has the universal property of \(\Cst(\Gamma; \B)\), and hence the two \(\Cst\)\nb-algebras \(\Cst(G;\A)\) and \(\Cst(\Gamma; \B)\) are isomorphic.
 	 \end{proof}
 	 \begin{remark}
 	 	Let \(p\colon \A\to G\) be a Fell bundle over an \emph{amenable} \etale\ groupoid \(G\) (see~\cite{Ananthraman-Renault2000Amenable-gpd} for the definition of amenable groupoid) with a continuous \(1\)\nb-cocycle \(c\colon G \to \Z\). Assume that the cocycle \(c\) is \emph{unperforated} in the sense of~\cite{Hossain-2026-Fell-bundle-alg-as-Cuntz-Pimsner}, that is, in some sense \(c^{-1}(e)\) generates~\(G\). Then~\cite[Theorem 3.9]{Hossain-2026-Fell-bundle-alg-as-Cuntz-Pimsner} shows that \(\Cst(G;\A)\) can be realized as the  Cuntz--Pimsner algebra of a \(\Cst\)\nb-correspondence constructed from the cocycle date~\(c\). Hence, Theorem~\ref{thm-Fell-bund-as-graded-alg} yields a topological grading of the Cuntz--Pimsner algebra constructed in~\cite{Hossain-2026-Fell-bundle-alg-as-Cuntz-Pimsner}.
 	 \end{remark}

 \section{Fell bundle algebras as coaction crossed products}\label{sec-coaction}
 
 In this section, we exploit the topological grading (equivalently, the group Fell bundle structure) of the section algebra \(\Cst(G;\A)\) established in Section~\ref{sec-Fell-bund-as-grad} to construct a coaction of the underlying discrete group on \(\Cst(G;\A)\). We then identify the crossed product by this coaction with the \(\Cst\)\nb-algebra of a Fell bundle over the \emph{skew-product groupoid} (see Theorem~\ref{thm-coaction-crossed-product}). For background on coactions of discrete groups and their crossed products, we refer the reader to Section~\ref{sec-cocation-dis}.

 \begin{proposition}\label{prop-coaction}
 Let \(p\colon \A\to G\) be a Fell bundle over a locally compact Hausdorff second countable groupoid equipped with a Haar system. Let \(\Gamma\) be a discrete group with a continuous \(1\)\nb-cocycle \(c\colon G\to \Gamma\). Then there exists a coaction \(\delta\) of \(\Gamma\) on \(\Cst(G;\A)\) such that 
 \begin{equation}\label{eq-coaction-prop-1}
 \delta(f_{\gamma}) = f_{\gamma}\otimes \gamma
\end{equation}
 where \(\supp(f_{\gamma}) \subseteq G_{\gamma}\).
 \end{proposition}
 \begin{proof}
 	Note that the group Fell bundle algebra \(\Cst(\Gamma, \B)\) is isomorphic to \(\Cst(G;\A)\) by Theorem~\ref{thm-Fell-bund-as-graded-alg}. Thus, the map \(\delta \colon \Cc(\Gamma, \B) \to \Cst(\Gamma, \B)\otimes \Cst(\Gamma) \) given by Equation~\eqref{eq-coaction-prop-1} gives a \(^*\)\nb-homomorphism. Since the enveloping \(\Cst\)\nb-algebra of \(\Cc(\Gamma; \B)\) is \(\Cst(\Gamma; \B) \cong \Cst(G;\A)\), \(\delta\) extends to a unique \(^*\)\nb-homomorphism \(\delta \colon \Cst(G;\A) \to \Cst(G;\A)\otimes \Cst(\Gamma)\). We now verify the coaction identity on the generating elements~\(f_{\gamma}\): 
 	\begin{multline*}
 		(\delta\otimes \id_{G})\circ \delta (f_\gamma) = 	(\delta\otimes \id_{G})(f_{\gamma} \otimes \gamma) = (f_{\gamma}\otimes \gamma)\otimes \gamma = f_{\gamma} \otimes (\gamma\otimes \gamma) \\
 		=\id_{\Cst(G;\A)}\otimes \delta_{G}(f_\gamma \otimes \gamma)
 		= (\id_{\Cst(G;\A)} \otimes \delta_{G}) \circ \delta (f_\gamma).
 	\end{multline*}
 	Therefore, the coaction identity holds on \(\Cst(G;\A)\).
 	Choose an approximate identity \((u_\alpha)_{\alpha}\) of \(\Cst(G_e;\A|_{G_e})\) in \(\Cc(G_e;\A|_{G_e})\), then we have
 	\[
 	 \delta(f_\gamma) (u_\alpha\otimes \gamma^{-1}\eta) = f_{\gamma}u_{\alpha} \otimes \eta \to f_{\gamma}\otimes \eta \quad \textup{as } \alpha \to \infty.
 	\]
 	Therefore, the map \(\delta\) is nondegenerate in the sense 
 	\[
 	 \overline{\textup{span}} \{\delta(\Cst(G;
 	 \A))(1 \otimes \Cst(\Gamma))\} = \Cst(G;\A) \otimes \Cst(\Gamma),
 	\] 
 	where \(1\) is the unit of the multiplier algebra \(M(\Cst(G;\A))\). 
 	Let \(1_G\colon G \to \C\) be the trivial representation of \(G\). Then for \(f_{\gamma} \in \Cc(G;\A)\), we have 
 	\[
 	 (\id_{\Cst(G;\A)}\otimes 1_G)\circ \delta(f_\gamma) =  (\id_{\Cst(G;\A)}\otimes 1_G)(f_{\gamma} \otimes \gamma) = f_{\gamma}\otimes 1 = f_{\gamma} =\id_{\Cst(G;\A)}(f_{\gamma}).
 	 \]
 	 Thus, we have \((\id_{\Cst(G;\A)}\otimes 1_G)\circ \delta =\id_{\Cst(G;\A)}\). Hence,~\cite[Lemma A.24]{Echterhoff-Kaliszewski-Quigg-Raeburn-book-categorical-appro-to-imprimitivity-thm} ensures that \(\delta\) is injective.
  \end{proof}
  Let \(c\colon G\to \Gamma\) be a continuous \(1\)\nb-cocycle. We now consider the skew-product groupoid from~\cite{Renault1980A-gpd-appr-to-cst-alg}. However, for skew-product, we follow the convention of~\cite{Kaliszewski-Quigg-Raeburn-2001-Skew-prod-cross-prod-by-coaction} rather than that of~\cite[Definition I.1.6]{Renault1980A-gpd-appr-to-cst-alg}. Nevertheless, these conventions yield isomorphic groupoids (see~\cite[Remark 4.1]{Kaliszewski-Quigg-Raeburn-2001-Skew-prod-cross-prod-by-coaction}).
   
  The underlying set of the \emph{skew-product groupoid} \(G(c)\) is \(G\times \Gamma\). The groupoid operations are as follows: \((g,\gamma)\) and \((h, \eta)\) are composable if and only if \((g,h) \in G^{(2)}\) and \(\gamma = c(h)\eta\) and 
  \[
   (g,c(h)\eta)(h, \eta) = (gh, \eta), \quad (g,\gamma)^{-1} =(g^{-1}, c(g)\gamma).
  \]
  The range and source maps are given by \(r(g,\gamma) = (r(g), c(g)\gamma)\) and \(s(g, \gamma) = (s(g), \gamma)\). Note that \(G(c)\) is a locally compact Hausdorff groupoid. Moreover,  if \(G\) has a Haar system, then \(G(c)\) also has a Haar system given by 
  \[
   \int f\dd\lambda^{(u,\gamma)} = \int_{G^{u}} f(g, c(g)^{-1}\gamma) \dd\lambda^u(x).
  \]  
  If \(G\) is an \etale\ groupoid, then \(G(c)\) is also \'etale. Define a groupoid homomorphism \(\phi\colon G(c) \to G\) by \(\phi(g, \gamma) =g\) for \((g,\gamma) \in G(c)\).
  
  Let \(p\colon \A\to G\) be a Fell bundle and \(c\colon G \to \Gamma\) a continuous  \(1\)\nb-cocycle. Consider the \emph{pull-back} bundle \(q\colon \phi^*\A\to G(c)\) over \(G(c)\) along the homomorphism \(\phi\) (recall the pull-back Fell bundle from Remark~\ref{rmk-pull-back})
   \[
  \begin{tikzcd}
  	\phi^*\A \arrow[r, dashed] \arrow[d, "q"]
  	& \A \arrow[d, "p"] \\
  	G(c) \arrow[r, "\phi"]
  	&  G.
  \end{tikzcd}
  \]
  Our main result in this section (Theorem~\ref{thm-coaction-crossed-product}) establishes an isomorphism between the coaction crossed product \(\Cst(G;\A)\rtimes_{\delta} \Gamma\) and the Fell bundle algebra \(\Cst(G(c); \phi^*\A)\).
  
 \begin{theorem}\label{thm-coaction-crossed-product}
 Let \(p\colon \A\to G\) be a Fell bundle over a locally compact Hausdorff second countable groupoid equipped with a Haar system. Let \(\Gamma\) be a discrete group with a continuous \(1\)\nb-cocycle \(c\colon G\to \Gamma\). Consider the pull-back Fell bundle as described above, and the coaction~\(\delta\) as in Proposition~\ref{prop-coaction}. Then the associated coaction crossed product \(\Cst(G;\A)\rtimes_{\delta}\Gamma\) is isomorphic to \(\Cst(G(c);\phi^*\A)\).
 \end{theorem}
 \begin{proof}
 	Recall the group Fell bundle \(\B\) over \(\Gamma\) from Theorem~\ref{thm-Fell-bund-as-graded-alg}. Let \(\Gamma\ltimes \Gamma\) be the groupoid  with underlying set \(\Gamma\times \Gamma\) and the operations are given by 
 	\((\gamma, \eta\zeta) (\eta, \zeta) = (\gamma\eta, \zeta)\) and \((\gamma, \eta)^{-1} = (\gamma^{-1}, \gamma\eta)\). Consider the product Fell bundle \(\B\times \Gamma\) over \(\Gamma\ltimes \Gamma\) whose fiber at \((\gamma, \eta)\) is \(\B_{\gamma}\times \{\eta\}\) and the operations on \(\B\times \Gamma\) are given by 
 	\[
 	 (f_{\gamma}, \eta \zeta)(g_{\eta}, \zeta) = (f_{\gamma}g_{\eta}, \zeta) \quad \textup{and} \quad (f_{\gamma}, \eta)^* = (f^*_{\gamma}, \gamma\eta)
 	\]
 	where \(f_{\gamma} \in \B_{\gamma}, g_{\eta} \in \B_{\eta}\) and \(\gamma, \eta, \zeta \in \Gamma\). Our next claim is to show that \(\Cst(G;\A)\rtimes_{\delta}\Gamma\) is the enveloping \(\Cst\)\nb-algebra of  \(\Cc(\Gamma\ltimes \Gamma; \B\times \Gamma)\). 
 	 Since \(\Cst(G;\A)\) is the enveloping \(\Cst\)\nb-algebra of the group Fell bundle \(\B\) (by Theorem~\ref{thm-Fell-bund-as-graded-alg}). Using~\cite[Theorem 3.3]{Echterhoff-Quigg-Induced-coaction-of-disctete-groups} it is enough to show that the unit fiber algebra \(\Cst(G;\A)_{e}\) of the Fell bundle associated to the coaction \(\delta\) is isomorphic to the \(\Cst\)\nb-algebra \(\B_e\). But note that \(\Cst(G;\A)_{e} = \{f\in \Cst(G;\A) : \delta(f)  =f\otimes e\}\) is the closure of \(\Cc(G_e;\A|_{G_e})\) inside \(\Cst(G;\A)\). Thus, Theorem~\ref{thm-isometric-isomorphism} ensures that \(\Cst(G;\A)_{e}\) is isomorphic to \(\B_e\).
 	 This proves the claim. 
 	 
 	 We now show that \(\Cst(\Gamma\ltimes \Gamma; \B\times \Gamma)\) is isomorphic to \(\Cst(G(c); \phi^*\A)\), \(\Cst\)\nb-algebra of the pull-back Fell bundle over \(G(c)\).   
 	 For \(\gamma, \eta \in \Gamma\), set 
 	 \[
 	  D_{\gamma, \eta} = \big\{f\in \Cc(G(c);\phi^*\A) : \supp(f) \subseteq G_{\gamma}\times \{\eta\}\big\}.
 	 \]
 	 Since \(\Gamma\) is discrete, the sets \(G_\gamma\times\{\eta\}\) are clopen in \(G(c)\). Every compact subset of \(G(c)\) intersects only finitely many such sets, and therefore every compactly supported section decomposes as a finite sum of elements supported in individual \(G_\gamma\times\{\eta\}\). Thus,
 	 \[
 	 \Cc(G(c);\phi^*\mathcal A)
 	 = \textup{span}\{D_{\gamma,\eta}:\gamma,\eta\in\Gamma\}.
 	 \]
      Define a linear map 
 	 \(\Phi\colon \Cc(\Gamma\ltimes \Gamma ;\B\times \Gamma)\to \Cc(G(c);\phi^*\A)\) by 
 	 \[
 	 \Phi(f,\gamma)(g,\eta) =  \begin{cases}
 	 	f(g) & \textup{if } \gamma  =\eta,\\
 	 	0 & \textup{ otherwise. }
 	 \end{cases}
 	 \]
 	 Then \(\Phi\) can be extended to a linear bijection from \(\Cc(\Gamma\ltimes \Gamma ;\B\times \Gamma)\to \Cc(G(c);\phi^*\A)\). We now show that it preserves the multiplication and involution. Fix \(\gamma, \eta, \zeta, \xi, \alpha \in \Gamma\) and \(f_{\gamma} \in \B_{\gamma}, g_{\zeta}\in \B_{\zeta}\) and \(h\in G\). Then we have 
 	 \begin{multline*}
 	 	 \Phi(f_{\gamma}, \eta)*\Phi(g_{\zeta}, \xi)(h, \alpha) \\= \int_{G(c)^{r(h, \alpha)}}  \Phi(f_{\gamma}, \eta) (k,\beta)\Phi(g_{\zeta}, \xi)((k, \beta)^{-1}(h, \alpha))\dd\lambda^{(r(h), c(h)\alpha)} (k,\beta).
 	 \end{multline*}
 	 Now \((k, \beta) \in G(c)^{r(h, \alpha)}\) if and only if \(r(k) = r(h)\) and \(c(k)\beta = c(h)\alpha\), that is, \(\beta = c(k^{-1}h)\alpha\). Thus, the last integral become
 	 \begin{equation}\label{eq-Main-thm-mult-comp-1}
 	 \int_{G^{r(h)}}  \Phi(f_{\gamma}, \eta) (k,c(k)^{-1}c(h)\alpha)\Phi(g_{\zeta}, \xi)(k^{-1}h, \alpha)\dd\lambda^{r(h)} (k).
 	\end{equation}
 	Using the definition of \(\Phi\), we have 
 	\[
 	 \Phi(f_{\gamma}, \eta) (k,c(k)^{-1}c(h)\alpha) = \begin{cases}
 	 	f_{\gamma}(k) & \textup{if } c(k)^{-1}c(h)\alpha  =\eta,\\
 	 	0 & \textup{ otherwise }
 	 	 \end{cases}
 	\]
 	and 
 	\[
 	\Phi(g_{\zeta}, \xi)(k^{-1}h, \alpha)  = \begin{cases}
 		g_{\zeta}(k^{-1}h) & \textup{if } \alpha  =\xi,\\
 		0 & \textup{ otherwise. }
 	\end{cases}
 	\]
 	Since \(\supp(f_{\gamma})\subseteq G_{\gamma}\) and \(\supp(g_{\zeta})\subseteq G_{\zeta}\), the integral in Equation~\eqref{eq-Main-thm-mult-comp-1} is nonzero only when \(c(k) = \gamma\) and \(c(k^{-1}h) = \zeta\), that is, \(c(h) =s(\zeta)\) and \(c(k)^{-1}c(h)\xi=\zeta\xi\). This implies that \(\eta =\zeta\xi\). Therefore, the integral in Equation~\eqref{eq-Main-thm-mult-comp-1} can be written as 
 	\begin{align*}
 	&\begin{cases}
 		\int_{G^{r(h)}} f_{\gamma}(k)g_{\zeta}(k^{-1}h)\dd\lambda^{r(h)}(k) & \textup{if } \alpha=\xi \textup{ and } \eta = \zeta\xi,\\
 		0 & \textup{ otherwise }
 	\end{cases}\\
 	=& \begin{cases}
 	f_{\gamma}*g_{\zeta}(h)
 	 & \textup{if } \alpha=\xi \textup{ and } \eta = \zeta\xi,\\
 		0 & \textup{ otherwise }
 	\end{cases}\\
 		=& \begin{cases}
 		\Phi(f_{\gamma}*g_{\zeta},\xi)(h,\alpha)
 		& \textup{if } \alpha=\xi \textup{ and } \eta = \zeta\xi,\\
 		0 & \textup{ otherwise }
 	\end{cases}\\
 	=&\Phi((f_{\gamma}, \eta)(g_{\zeta}, \xi))(h,\alpha).
 		\end{align*}
 		For involution, we have 
 		\begin{align*}
 		(\Phi(f_{\gamma}, \eta))^*(h,\zeta) &= \bigr(\Phi(f_\gamma, \eta)(h,\zeta)^{-1}\bigr)^* = \bigr(\Phi(f_\gamma, \eta)(h^{-1}, c(h)\zeta)\bigr)^*\\
 		&= \begin{cases}
 			\big(f_{\gamma}(h^{-1})\big)^*
 			& \textup{if } t=c(h)\zeta,\\
 			0 & \textup{ otherwise}
 		\end{cases}\\
 			&= \begin{cases}
 			f^*_{\gamma}(h)
 			& \textup{if } \gamma\eta=\zeta,\\
 			0 & \textup{ otherwise }
 		\end{cases}\\
 		&=\Phi(f^*_{\gamma}, \gamma\eta)(h,\zeta)\\
 		& = \Phi\big(f_{\gamma}, \eta\big)^*(h,\zeta).
 		\end{align*}
 		Thus, \(\Phi\) is a \(^*\)\nb-algebra isomorphism from \(\Cc(\Gamma\ltimes \Gamma ;\B\times \Gamma)\to \Cc(G(c);\phi^*\A)\). Therefore, \(\Phi\) gives a bijection between the nondegenerate representations of \(\Cc(\Gamma\ltimes \Gamma ;\B\times \Gamma)\) and \(\Cc(G(c);\phi^*\A)\) and hence \(\Phi\) can be extended to be an isomorphism \(
        \Cst(\Gamma\ltimes\Gamma;\B\times\Gamma)
 		\cong
 		\Cst(G(c);\phi^*\A)\).
 		Therefore, the coaction crossed product \(\Cst(G;\A)\rtimes_{\delta}\Gamma\) is isomorphic to \(\Cst(G(c);\phi^*\A)\). 
 \end{proof}

 \section{Applications}\label{sec-application}

  \subsection{Twisted groupoid \(\Cst\)-algebras}
  Let \(G\) be a locally compact Hausdorff second countable groupoid equipped with a Haar system and \(\sigma \in Z^2(G,\mathbb{T})\). Let \(\A = G\times \C\) equip with the product topology. Define \(p\colon \A\to G\) by \(p(g, z) =g\). Then \(\A\) is a Fell bundle, called \emph{Fell line bundle}, with respect to the operations
  \[
   (g,z)(h,w) =(gh, \sigma(g,h)zw) \quad \textup{and} \quad (g,z)^*=(g^{-1}, \overline{\sigma(g,g^{-1})z}).
  \]
Then the Fell bundle algebra \(\Cst(G;\A)\) is isomorphic to the twisted groupoid \(\Cst\)\nb-algebra \(\Cst(G,\sigma)\) (see~\cite[Lemma 4.1]{Afsar-Sims2021KMS-state-on-Cst-alg-Fell-bundle-over-gpd}). Let \(c\colon G \to \Gamma\) be a continuous \(1\)\nb-cocycle, where \(\Gamma\) is a discrete group. The restriction of \(\sigma\) on \(G_e\) gives a \(\mathbb{T}\)\nb-valued \(2\)-cocycle on \(G_e\), and the corresponding Fell line bundle over \(G_e\) can be identified with the restricted Fell bundle~\(\A|_{G_e}\). Therefore, applying Theorem~\ref{thm-isometric-isomorphism} on the Fell line bundle and identifying the Fell bundle \(\Cst\)\nb-algebras with the corresponding twisted groupoid \(\Cst\)\nb-algebras, we obtain the following corollary.
\begin{corollary}\label{coro-twisted-gpd-alg}
Let \(G\) be a locally compact Hausdorff second countable groupoid equipped with a Haar system. Let \(\sigma \in Z^2(G,\mathbb{T})\) and \(c\colon G\to \Gamma\) a continuous \(1\)\nb-cocycle. The  map \(\iota\colon  \Cc(G_e, \sigma) \to \Cc(G, \sigma)\) defined by
\[
\iota(f)(\gamma) = \begin{cases}
	f(\gamma) & \textup{if } \gamma \in G_e,\\
	0 & \textup{ otherwise }
\end{cases}
\]
extends to an injective homomorphism from \(\Cst(G_e,\sigma) \to \Cst(G,\sigma)\). 
\end{corollary}
\noindent Corollary~\ref{coro-twisted-gpd-alg} extends Theorem~1.1 of~\cite{Armstrong-Clark-an-Huef-Inclusion-graded-gpd} to the setting of twisted groupoid \(\Cst\)\nb-algebras. 

\begin{proposition}\label{prop-coation-twisted-ver}
Let \(G\) be a locally compact Hausdorff second countable groupoid equipped with a Haar system. Let \(\sigma \in Z^2(G,\mathbb{T})\) and let \(\Gamma\) be a discrete group with a continuous \(1\)\nb-cocycle \(c\colon G\to \Gamma\). Then there exists a coaction~\(\delta\) of \(\Gamma\) on \(\Cst(G,\sigma)\) and the associated coaction crossed product \(\Cst(G,\sigma)\rtimes_{\delta}\Gamma\) is isomorphic to \(\Cst(G(c),\sigma)\).
\end{proposition}
\begin{proof}
Consider the Fell line bundle \(\A=G\times \C\) over \(G\) associate to \(\sigma\). Proposition~\ref{prop-coaction} ensures that there exists a coaction of \(\Gamma\) on \(\Cst(G;\A)\) and hence on \(\Cst(G,\sigma)\). Moreover, Theorem~\ref{thm-coaction-crossed-product} allows us to identify the coaction crossed product \(\Cst(G,\sigma)\rtimes_{\delta}\Gamma\) with \(\Cst(G(c); \phi^*\A)\), where \(\phi^*\A\) is the pull-back Fell bundle over the skew-product \(G(c)\). Note that the \(2\)\nb-cocycle \(\sigma\) can be lifted to the skew-product as follows: \(\widetilde{\sigma}\colon G(c)^{(2)}\to \mathbb{T}\) by \(\widetilde{\sigma}((g,\gamma), (h, \eta)) =\sigma(g,h)\). And the pull-back bundle \(\phi^*\A\) is a Fell line bundle as \(\phi^*\A_{(g,\gamma)} \cong \A_{\phi(g,\gamma)} = \A_g\cong \C\) and the operations are given by 
\[
((g,\gamma),z)((h,\eta), w) = ((g,\gamma)(h,\eta), \sigma(g,h)zw)
\] 
and 
\[
 ((g,\gamma),z)^*=((g, \gamma)^{-1}, \overline{\sigma(g,g^{-1})z}).
\]
Therefore, this pull-back bundle \(\phi^*\A\) can be identify the Fell line bundle over \(G(c)\) associated to the twist (2-cocycle) \(\widetilde{\sigma}\). Hence, we have 
\[
 \Cst(G,\sigma)\rtimes_{\delta}\Gamma \cong \Cst(G(c);\phi^*\A)\cong\Cst(G(c), \widetilde{
 \sigma}) =  \Cst(G(c),
 \sigma).
\]
\end{proof}
\begin{remark}
Proposition~\ref{prop-coation-twisted-ver} applies to \(\sigma=1\) recovers a result of Kaliszewski, Quigg and Raeburn in~\cite[Theorem 4.3]{Kaliszewski-Quigg-Raeburn-2001-Skew-prod-cross-prod-by-coaction}. This proposition gives a \emph{non-\'etale} version of their result.  
\end{remark}
  
\subsection{Twisted higher-rank graph \(\Cst\)-algebras}
Higher-rank graphs or \(k\)\nb-graphs are introduced by Kumjian and Pask in~\cite{Kumjian-Pask2000Higher-Rank-graph-cst-alg} as a higher dimensional generalization of Cuntz--Krieger algebras. A \emph{\(k\)-graph} is a countable category \(\Lambda\) together with a functor \(d\colon \Lambda \to \N^{k}\) satisfying the factorization property, that is, for every \(\lambda \in \Lambda\) with \(d(\lambda) = m+n\) there are unique elements \(\mu, \gamma \in \Lambda\) such that 
\[
s(\mu) =r(\gamma),\quad  d(\mu) = m, \quad d(\gamma) =n \quad \textup{and}\quad  \lambda = \mu\gamma.
\] The set of vertices is denoted by \(\Lambda^{0}\) and can be identified with \(d^{-1}(0)\). 

Let \(\Lambda^2=\{(\lambda, \mu) \in \Lambda \times \Lambda : s(\lambda) =r(\mu)\}\). A map  \(\sigma\colon \Lambda^2 \to \mathbb{T}\) is called a \(\mathbb{T}\)\nb-valued \(2\)\nb-cocycle on \(\Lambda\) if \(\sigma(r(\lambda),\lambda) =\sigma(\lambda, s(\lambda))=1\) and 
\[
 \sigma(\lambda, \mu)\sigma(\lambda\mu, \nu) = \sigma(\lambda, \mu\nu)\sigma(\mu, \nu)
\]
for all composable triples \((\lambda, \mu,\nu)\). The set of all \(\mathbb{T}\)\nb-valued \(2\)\nb-cocycle is denoted by \(Z^2(\Lambda, \mathbb{T})\). For a row finite \(k\)\nb-graph with a \(2\)\nb-cocycle \(\sigma\), the \emph{twisted \(k\)-graph \(\Cst\)\nb-algebra} will be denoted by \(\Cst(\Gamma, \sigma)\) (see~\cite[Section 5]{Kumjian-Pask-Sims2015Twisted-Higher-rank-graph-alg} for a definition of twisted \(k\)-graph algebra).

 Let \(\Omega_k = \{(m,n)\in \N^k\times \N^k : m\leq n\}\). Then \(\Omega_k\) forms a \(k\)-graph, where the structure maps are given by \(r(m,n) = (m,m), s(m,n) =(n,n)\), \((m,n)(n,k) = (m,k)\) and \(d(m,n) = n-m\). The vertex set \(\Omega^0_k\) can be identified with \(\N^k\).
 
The \emph{infinite-path space} of a \(k\)\nb-graph \(\Lambda\) is given by
\[
\Lambda^{\infty} = \{x\colon \Omega_k \to \Lambda \mid x \textup{ is a functor that intertwines the degree maps}\}.
\]
For \(l\in \N^{k}\), the shift map \(\rho^l\colon \Lambda^{\infty} \to \Lambda^{\infty}\) is defined by 
\[
\rho^l(x)(m,n) \defeq x(m+l, n+l).
\]
 Let \(\Lambda\) be a row finite \(k\)-graph with no source. Then
 \[
 	G_{\Lambda} :=\big\{(x,l,y) \in \Lambda^{\infty}\times \Z^k\times \Lambda^{\infty} : l=m-n, m,n\in \N^k
 \textup{ and } \rho^m(x) =\rho^n(y)\big\}
 \] 
is a groupoid, called \emph{infinite-path groupoid}. The structure maps are given by 
\(r(x,l,y)= x, s(x,l,y) =y, (x,l,y) (y,l^{\prime}, z) = (x, l+l^{\prime}, z)\) and \((x,l,y)^{-1} = (y,-l,x)\). 
The unit space of \(G_{\Lambda}\) can be identified with \(\Lambda^{\infty}\). For \(\lambda, \mu \in \Lambda\) satisfying \(s(\lambda) = r(\mu)\), define
 \[
 Z(\lambda,\mu) \defeq \big\{(\lambda x,\, d(\lambda) - d(\mu),\, \mu x) \in \mathcal{G}_\Lambda : x \in \Lambda^\infty \textup{ with } r(x) = s(\lambda)\big\}.
 \]
 The collection \(\{Z(\lambda,\mu) : \lambda, \mu \in \Lambda\}\) forms a basis for a locally compact Hausdorff topology on \(G_\Lambda\).  Moreover, \(G_{\Lambda}\) is an \'etale groupoid (\cite[Proposition 2.8]{Kumjian-Pask2000Higher-Rank-graph-cst-alg}).

Let \(
\Lambda{_s *_s} \Lambda := \{(\lambda,\mu) \in \Lambda \times \Lambda : s(\lambda)=s(\mu)\}.\)
Fix a subset \(\mathcal{P} \subseteq \Lambda {_s *_s} \Lambda\) such that \((\lambda,s(\lambda)) \in \mathcal{P}\) for all \(\lambda \in \Lambda\) and
\[
G_{\Lambda} = \bigsqcup_{(\lambda,\mu)\in \mathcal{P}} Z(\lambda,\mu).
\]
Such a choice is always possible by \cite[Lemma~6.6]{Kumjian-Pask-Sims2015Twisted-Higher-rank-graph-alg}. For \(\alpha \in G_{\Lambda}\), let \((\lambda_\alpha,\mu_\alpha)\) denote the unique element of \(\mathcal{P}\) with \(\alpha \in Z(\lambda_\alpha,\mu_\alpha)\), and define \(f\colon G_{\Lambda} \to \mathbb{Z}^k\) by \(f(x,n,y)=n\). Let \(\omega\) be a \(2\)-cocycle on \(\Lambda\). Then \cite[Lemma~6.3]{Kumjian-Pask-Sims2015Twisted-Higher-rank-graph-alg} ensures that for any composable pair \((\alpha,\beta) \in G^{(2)}_{\Lambda}\), there exist \(\nu,\xi,\zeta \in \Lambda\) and \(y \in \Lambda^\infty\) satisfying \(
\mu_\alpha \nu = \lambda_\beta \xi, 
\lambda_\alpha \nu = \lambda_{\alpha\beta} \zeta,
\mu_\beta \xi = \mu_{\alpha\beta} \zeta\),
and
\[
\alpha = (\lambda_\alpha \nu y,f(\alpha),\mu_\alpha \nu y),\quad
\beta = (\lambda_\beta \xi y,f(\beta),\mu_\beta \xi y),\quad
\alpha\beta = (\lambda_{\alpha\beta} \zeta y, f(\alpha\beta), \mu_{\alpha\beta} \zeta y).
\]
Then the map \(\sigma_{\omega}\colon G^{(2)}_{\Lambda} \to \mathbb{T}\) by
\begin{equation}\label{eq-2-cocycle-gpd}
\sigma_{\omega}(\alpha,\beta)
= \omega(\lambda_\alpha,\nu)\overline{\omega(\mu_\alpha,\xi)}\omega(\lambda_\beta,\xi)\overline{\omega(\mu_\beta,\xi)}\overline{\omega(\lambda_{\alpha\beta},\zeta)}\omega(\mu_{\alpha\beta},\zeta)
\end{equation}
defines a continuous \(2\)-cocycle on \(G_{\Lambda}\), independent of the choice of \(\nu,\xi,\zeta\). Moreover, by \cite[Theorem~6.5]{Kumjian-Pask-Sims2015Twisted-Higher-rank-graph-alg}, the cocycles arising from different choices of \(\mathcal{P}\) are cohomologous. And the twisted \(k\)-graph \(\Cst\)\nb-algebra \(\Cst(\Lambda, \omega)\) is isomorphic to the twisted groupoid algebra \(\Cst(G_{\Lambda}, \sigma_{\omega})\) (by~\cite[Corollary 7.8]{Kumjian-Pask-Sims2015Twisted-Higher-rank-graph-alg}).

 A \(\Gamma\)\nb-valued \(1\)\nb-cocycle \(\eta\) on \(\Lambda\) is a functor \(\eta \colon \Lambda \to \Gamma\). Then the skew-product \(k\)-graph \(\Lambda\times_{\eta}\Gamma\) is defined by \((\Lambda\times_{\eta}\Gamma)^0 = \Lambda^0\times \Gamma,\) where \(r,s\colon \Lambda\times_{\eta} \Gamma \to (\Lambda\times_{\eta} \Gamma)^0\) by \(r(\lambda, t) = (r(\lambda), t), s(\lambda,t) = (s(\lambda), \eta(\lambda)t)\) and \(d\colon \Lambda\times_{\eta} \Gamma \to \N^k\) by \(d(\lambda, t) = d(\lambda)\). The map \(\eta\) induces a \(1\)-cocycle \(c_\eta\colon G_{\Lambda} \to \Gamma\) by 
\begin{equation}\label{eq-cocycle-gpd}
c_{\eta}(x,m-n,y) = \eta(x(m,0))\eta(y(n,0))^{-1}.
\end{equation}

\begin{proposition}\label{prop-identfy-skew-prod-coaction}
	Let \(\Lambda\) be a row-finite \(k\)\nb-graph with no sources and let \(\omega \in Z^2(\Lambda,\mathbb{T})\). Let \(\eta\colon \Lambda \to \Gamma\) be a \(1\)\nb-cocycle, where \(\Gamma\) is a discrete group. Then there exists a coaction \(\delta\) of \(\Gamma\) on \(\Cst(\Lambda,\omega)\). Moreover,
	\[
	\Cst(\Lambda,\omega)\rtimes_\delta \Gamma \;\cong\; \Cst(\Lambda\times_\eta \Gamma,\widetilde{\omega}),
	\]
	where \(\widetilde{\omega}\) is the \(2\)\nb-cocycle on \(\Lambda\times_\eta \Gamma\) defined by \(
	\widetilde{\omega}((\lambda,g),(\mu,\eta(\lambda)g)) := \omega(\lambda,\mu).\)
\end{proposition}

\begin{proof}
	By \cite[Corollary~7.8]{Kumjian-Pask-Sims2015Twisted-Higher-rank-graph-alg}, there exists an isomorphism
	\[
	\Cst(\Lambda,\omega) \;\cong\; \Cst(G_\Lambda,\sigma_\omega),
	\]
	where \(\sigma_\omega\) is the groupoid \(2\)\nb-cocycle defined by  Equation~\eqref{eq-2-cocycle-gpd}.
	The \(1\)\nb-cocycle \(\eta\colon \Lambda \to \Gamma\) induces a continuous groupoid \(1\)-cocycle \(c_\eta\colon G_\Lambda \to \Gamma\) given by Equation~\eqref{eq-cocycle-gpd}. Hence, by Proposition~\ref{prop-coation-twisted-ver}, there exists a coaction \(\delta\) of \(\Gamma\) on \(\Cst(G_\Lambda,\sigma_\omega)\) such that
	\[
	\Cst(G_\Lambda,\sigma_\omega)\rtimes_\delta \Gamma \;\cong\; \Cst(G_\Lambda(c_\eta),\sigma_\omega),
	\]
	where \(G_\Lambda(c_\eta)\) is the skew-product groupoid.
	
	Our next claim is to establishes an isomorphism of twisted groupoid \(\Cst\)\nb-algebras
	\begin{equation}\label{eq-identifiaction-gpd-skew-prd}
	\Cst(G_\Lambda(c_\eta),\sigma_\omega) \;\cong\; \Cst(G_{\Lambda\times_\eta \Gamma},\sigma_{\widetilde{\omega}}).
	\end{equation}
	\noindent\emph{Proof of the claim.} We can identify the infinite-path space of \(\Lambda\times_{\eta}\Gamma\) with \(\Lambda^{\infty}\times \Gamma\). Under this identification, the shift map satisfies the following:
	\[
	\rho^r(x,g) = \big(\rho^r(x),\, g\,\eta(x(r,0))\big),
	\]
	for \(r\in \N^k\). Let \(((x,g),\, l,\, (y,h)) \in  G_{\Lambda\times_\eta \Gamma}\). Then \(	\rho^m(x,g)=\rho^n(y,h)\) for \(m,n\in\mathbb{N}^k\) with \(l=m-n\). This give us
	\[
	\rho^m(x)=\rho^n(y)
	\quad \text{and} \quad
	g\,\eta(x(m,0)) = h\,\eta(y(n,0)).
	\]
  The last equality gives us
	\[
	h = g\,\eta(x(m,0))\,\eta(y(n,0))^{-1}
	= g\,c_\eta(x,m-n,y).
	\]
	Therefore,
	\[
	G_{\Lambda\times_\eta \Gamma}
	=
	\Big\{((x,g),l,(y,h)) : (x,l,y)\in G_\Lambda,\ h = g\,c_\eta(x,l,y)\Big\}.
	\]
	Define \(
	\Phi: G_\Lambda(c_\eta) \to G_{\Lambda\times_\eta \Gamma}\)
	by
	\[
	\Phi((x,l,y),g)
	=
	\big((x,g),\,l,\,(y, g\,c_\eta(x,l,y))\big).
	\]
	Using the definition of \(c_\eta\), one checks that \(\Phi\) is a well-defined bijection preserving multiplication and inversion. Hence \(\Phi\) is a groupoid isomorphism.
	 Since the cocycle \(\widetilde{\omega}\) is defined by 
	\(\widetilde{\omega}((\lambda,g),(\mu,g\eta(\lambda)))=\omega(\lambda,\mu)\), the associated groupoid cocycle \(\sigma_{\widetilde{\omega}}\) depends only on the underlying \(\Lambda\)-paths. Therefore, \(\Phi\) intertwines the cocycles \(\sigma_\omega\) and \(\sigma_{\widetilde{\omega}}\) 
	and so induces an isomorphism
	\[
	\Cst(G_\Lambda(c_\eta),\sigma_\omega)
	\;\cong\;
	\Cst(G_{\Lambda\times_\eta \Gamma},\sigma_{\widetilde{\omega}}).
	\]
	This proves the claim. Combining the above identifications, we obtain
	\begin{align*}
		\Cst(\Lambda\times_\eta \Gamma,\widetilde{\omega})
		&\cong \Cst(G_{\Lambda\times_\eta \Gamma},\sigma_{\widetilde{\omega}}) \quad (\textup{by \cite[Corollary~7.8]{Kumjian-Pask-Sims2015Twisted-Higher-rank-graph-alg}})\\
		&\cong \Cst(G_\Lambda(c_\eta),\sigma_\omega) \quad \textup{ (by Equation~\eqref{eq-identifiaction-gpd-skew-prd})}\\
		&\cong \Cst(G_\Lambda,\sigma_\omega)\rtimes_\delta \Gamma \quad \textup{(by Proposition~\ref{prop-coation-twisted-ver})} \\
		&\cong \Cst(\Lambda,\omega)\rtimes_\delta \Gamma \quad \textup{(by \cite[Corollary~7.8]{Kumjian-Pask-Sims2015Twisted-Higher-rank-graph-alg})},
	\end{align*}
	which completes the proof.
\end{proof}
\begin{remark}
In particular, Proposition~\ref{prop-identfy-skew-prod-coaction} extends a result of Kaliszewski--Quigg--Raeburn~\cite[Theorem~2.4]{Kaliszewski-Quigg-Raeburn-2001-Skew-prod-cross-prod-by-coaction} from graph \(\Cst\)\nb-algebras to the broader setting of twisted higher-rank graph \(\Cst\)\nb-algebras.
\end{remark}

\paragraph{\itshape Acknowledgements:}
We thank Rohit Dilip Holkar and Ralf Meyer for fruitful discussions.


\end{document}